\documentclass[a4paper,10pt,final]{siamart220329}
\usepackage[english]{babel} 
\usepackage{csquotes} 
\usepackage{subcaption}
\usepackage{hyperref,url}
\usepackage{xcolor}
\usepackage{amssymb,amsfonts,amsmath}
\usepackage[notrig]{physics} 
\usepackage{braket} 
\usepackage{mleftright} 
\mleftright
\usepackage[compress,sort]{cite}
\usepackage{textcomp}
\usepackage{enumitem}

\newcommand*\new[1]{{#1}}

\headers{Warped geometries of Segre--Veronese manifolds}{S.~Jacobsson et al.}

\title{Warped geometries of Segre--Veronese manifolds\thanks{\textbf{Funding}: This project was funded by BOF project C16/21/002 by the Internal Funds KU Leuven and FWO project G080822N. J. Van der Veken is additionally supported by the Research Foundation---Flanders (FWO) and the Fonds de la Recherche Scientifique (FNRS) under EOS Project G0I2222N.}}
\author{%
\hspace{1cm}Simon Jacobsson\thanks{KU Leuven, Department of Computer Science, Celestijnenlaan 200A -- box 2402, B-3000 Leuven, Belgium (\email{simon.jacobsson@kuleuven.be}). ORCiD: 0000-0002-1181-972X} \and
Lars Swijsen\thanks{KU Leuven, Department of Mathematics, Celestijnenlaan 200B -- box 2400, B-3000, Leuven, Belgium. ORCiD: 0000-0002-8687-853X} \and
Joeri Van der Veken\thanks{KU Leuven, Department of Mathematics, Celestijnenlaan 200B -- box 2400, B-3000, Leuven, Belgium (\email{joeri.vanderveken@kuleuven.be}). ORCiD: 0000-0003-0521-625X} \and \hspace{1cm}
Nick Vannieuwenhoven\thanks{KU Leuven, Department of Computer Science, Celestijnenlaan 200A -- box 2402, B-3000 Leuven, Belgium (\email{nick.vannieuwenhoven@kuleuven.be}); Leuven.AI -- KU Leuven Institute for AI, B-3000 Leuven, Belgium. ORCiD: 0000-0001-5692-4163}}

\newsiamremark{remark}{Remark}
\newsiamthm{algprop}{Algorithm}

\DeclareMathAlphabet{\mathpzc}{OT1}{pzc}{m}{it}
\newcommand{\RR}{\mathbb{R}}
\newcommand{\Sp}{\mathbb{S}}

\newcommand{\Tang}{\mathrm{T}}
\newcommand{\deriv}{\mathrm{d}}

\newcommand{\vect}[1]{\mathbf{#1}}
\newcommand{\tensor}[1]{\mathpzc{#1}}
\newcommand{\dist}{\mathrm{dist}}
\newcommand{\Var}[1]{\mathcal{#1}}
\renewcommand{\innerproduct}[2]{\left\langle{}#1, #2\right\rangle}

\allowdisplaybreaks

\begin{document}

\maketitle

\begin{abstract}
	Segre--Veronese manifolds are smooth submanifolds of tensors comprising the partially symmetric rank-$1$ tensors.
	We investigate a one-parameter family of warped geometries of Segre--Veronese manifolds, which includes the standard Euclidean geometry. This parameter controls by how much spherical tangent directions are weighted relative to radial tangent directions.
	We present closed expressions for the exponential map, the logarithmic map, and the intrinsic distance in these warped Segre--Veronese manifolds, which can be computed efficiently numerically.
	It is shown that Segre--Veronese manifolds are not geodesically connected in the Euclidean geometry, while they are for some values of the warping parameter. The benefits of geodesic \new{connectedness} may outweigh using the Euclidean geometry in certain applications. One such application is presented: numerically computing the Riemannian center of mass for averaging rank-$1$ tensors.
\end{abstract}

\section{Introduction}

The $\vect{k}$th \emph{Segre--Veronese manifold} $\Var{S}^\vect{k}$ of \emph{partially symmetric rank-$1$ tensors} in the space of $n_1^{k_1} \times n_2^{k_2} \times\cdots\times n_d^{k_d}$ real arrays is a fundamental set of elementary tensors~\cite{Landsberg2012}. 
Special cases include the punctured plane 
\begin{align*}
	\Var{S}^{(1)} := \RR^n_* := \set{ \lambda \vect{u} \mid \lambda \in \RR_{*}, \norm{\vect{u}} = 1 };
\end{align*}
the manifold of symmetric rank-$1$ matrices
\begin{align*}
	\Var{S}^{(2)} := \set{ \lambda \vect{u} \vect{u}^{\mathsf{T}} \mid \lambda\in\RR_{*},\, \norm{\vect{u}}  = 1 } \subset \RR^{n \times n},
\end{align*}
and, more generally, the Veronese manifold $\Var{S}^{(d)}$ of symmetric rank-$1$ tensors \cite{Landsberg2012,CGLM2008}; and the manifold of rank-$1$ matrices
\begin{align*}
	\Var{S}^{(1,1)} := \set{ \lambda \vect{u} \vect{v}^{\mathsf{T}} \mid \lambda\in \RR_{*},\, \norm{\vect{u}} = 1,\, \norm{\vect{v}} = 1 } \subset \RR^{m \times n},
\end{align*}
and, more generally, the Segre manifold $\Var{S}^{(1,\ldots,1)}$ of rank-$1$ tensors \cite{Landsberg2012}. 
Analogously to the role of rank-$1$ matrices in a low-rank matrix decomposition, Segre--Veronese manifolds are the basic building blocks of more advanced tensor rank decomposition models \cite{Landsberg2012}. Such sums of partially symmetric rank-$1$ tensors feature in a plethora of applications in biomedical engineering, chemometrics, computer science, data science, machine learning, psychometrics, signal processing, and statistics, among others; see the overview articles \cite{KB2009,SdLFHPF2017,PFS2016,Anandkumar2014,Morup2011,BTYZQ2021} and the references therein.

The \emph{algebraic} geometry of the Segre--Veronese manifold and its secant varieties have been thoroughly studied since the $19$th century~\cite{BCCGO2018}. By contrast, comparatively little is known about the \emph{Riemannian} geometry of Segre--Veronese manifolds, let alone of (the smooth loci of) its higher secant varieties.
To our knowledge, only the Euclidean geometry of the Segre manifold of rank-$1$ tensors (i.e., $\vect{k}=(1, \dots ,1)$) as a submanifold of $\RR^{n_1 \times \dots \times n_d}$ was investigated in the prior works~\cite{SVV2022,Swijsen2022}.

The present paper studies essential properties of the Segre--Veronese manifold $\Var{S}^\vect{k}$ as a \emph{metric space}.
We are primarily interested in characterizing its \emph{geodesics}.
These curves $\gamma \subset \Var{S}^\vect{k}$ can be thought of as (locally) length-minimizing curves that generalize straight lines in Euclidean spaces.
\new{Being solutions of differential equations, geodesics can be characterized as the solution of an \emph{initial value} or a \emph{boundary value} problem:}
\begin{enumerate}
	\item A point $\tensor{P} \in \Var{S}^\vect{k}$ and a tangent vector $\dot{\tensor{P}} \in \Tang_\tensor{P} \Var{S}^{\vect{k}}$ \new{can be specified to identify a geodesic $\gamma$ that satisfies} $\gamma(0)=\tensor{P}$ and $\gamma'(0) := {\derivative{\gamma}{t}}(0) = \dot{\tensor{P}}$. \new{The map that takes these inputs to the corresponding geodesic is called the manifold's \emph{exponential map}.}
	\item Two \emph{``endpoints''} $\tensor{P}$, $\tensor{Q} \in \Var{S}^\vect{k}$ \new{can be specified to identify a geodesic $\gamma$ that} satisfies $\gamma(0)=\tensor{P}$ and $\gamma( \ell )=\tensor{Q}$, where $\ell$ is the distance between $\tensor{P}$ and $\tensor{Q}$. \new{The map that takes these two endpoints as input and sends them to the corresponding geodesic is called the \emph{logarithmic map} of the manifold.}
\end{enumerate}
\new{We will characterize the geodesics of Segre--Veronese manifolds by presenting explicit expressions for their exponential and logarithmic maps.}

\subsection{Applications of geodesics}
\new{The logarithmic map} finds application in (i) statistics on manifolds~\cite{Fletcher2020}, (ii) interpolation on manifolds~\cite{Zimmermann2021}, and (iii) approximation of maps into manifolds.
First, the definition and computation of statistics on manifolds often involves geodesics \new{connecting two points}~\cite{Fletcher2020}.
The simplest statistic, the \emph{Fr\'echet mean}~\cite{Frechet1948}, \emph{Riemannian center of mass}, or \emph{Karcher mean}~\cite{Karcher1977}, in a manifold $\Var{M}$ consists of the point(s) that minimize the integral of the squared intrinsic distances to a fixed set of points in $\Var{M}$.
The gradient of the squared intrinsic distance function from $x$ to a fixed point $x^\star\in\Var{M}$ is the tangent vector at $x$ of a geodesic connecting $x$ and $x^\star$. Therefore, geodesics appear for computing such a Fr\'echet mean both in gradient-based Riemannian optimization algorithms ~\cite{PFA2006,Pennec2006,JVV2012} and other approximation algorithms~\cite{Karcher1977,JVV2012,MHA2018,Chakraborty20}.
\new{The logarithmic map} also appears in the definition and computation of higher-order statistics on manifolds, such as \emph{principal geodesic analysis}~\cite{FLPJ2004,FJ2007,SLN2013} and the Gaussian distribution and its parameter estimation~\cite{Pennec2006,Chakraborty19}.
It is an essential ingredient in \emph{geodesic regression}~\cite{BA2011,Fletcher2012,Fletcher2020} and generalizations thereof~\cite{HFJ2014}.
Second, many algorithms for interpolation of data in manifolds rely on \new{the logarithmic map} \cite{NHP1989,MSL2006,JSLR2006,PN2007,NYP2013,GMA2018,SA2019,Zimmermann20,Zimmermann2021,MMHAM2021,ZK2022,ZB2024}.
Third, geodesics \new{connecting endpoints} are an essential building block of algorithms for constructing an approximation of a function from a Euclidean space into a Riemannian manifold. They are used in the Riemannian moving least squares algorithm~\cite{Grohs2012,GSY2016,SCW2023}, and the pullback-based methods of~\cite{LT2017,JVVV2024,WVVV2025}.

\new{The exponential map has} applications in (i) Riemannian optimization \cite{AMS2008,Boumal2023}, and (ii) integration on manifolds~\cite{HWL2006}.
First, \new{the exponential map} (or approximations thereof called \emph{retractions}~\cite{Seguin24}) are an essential component of \emph{Riemannian optimization algorithms}, which optimize a smooth objective function over a constraint set that is a smooth manifold~\cite{AMS2008,Boumal2023}.
Second, \new{the exponential map}, the projection retraction~\cite{AM2012}, or other retractions are essential in intrinsic integration schemes for vector fields on manifolds, such as the ones arising from differential equations on manifolds~\cite{CG1993,HM1994,M-KZ1997,M-K1998,Hairer2001,LO2002,HWL2006} or the projected ones from dynamic low-rank approximation methods (see~\cite{KL2007,KL2010,LRSV2013,UV2013,LOV2015,FL2018,CGV2023,HNS2023,CL2024}).

The foregoing discussion described general potential applications of geodesics. Next, we highlight plausible applications in the context of Segre--Veronese manifolds. Geodesics can be used in Riemannian optimization for seeking an approximate \emph{partially symmetric tensor rank decomposition}~\cite{BCCGO2018}. Such algorithms were already applied to approximate tensor rank decompositions using a retraction in~\cite{BV2018} and using \new{the exponential map} in~\cite{SVV2022}, and to \emph{symmetric tensor rank} or \emph{Waring decompositions}~\cite{BCCGO2018,CGLM2008} using a retraction in~\cite{KKM2022}.
As illustrated in \cref{sec_applications}, computing a Fr\'echet mean is a natural alternative to averaging tensor rank decompositions based on factor matrices as was proposed in~\cite{CST2023} to improve the robustness of stochastic tensor decompositions.

\subsection{Contributions}
Our main contributions are analytic expressions of the exponential and logarithmic maps of Segre--Veronese manifolds, in \cref{cor_exponential_map} and \cref{thm_logarithm_SV} respectively.
We establish a few auxiliary results, such as a formula for the distance between points in Segre--Veronese manifolds in \cref{prop:distance_formula}.

\Cref{prop:existence_of_minimizing_geodesics} entails that not every pair of partially symmetric rank-$1$ tensors in a Segre--Veronese manifold can be connected by a minimizing geodesic in the standard Euclidean geometry.
In some applications, such as the one in \cref{sec_applications}, it is more important that every pair of points is connected by a minimizing geodesic than that Euclidean geometry is used.
For this reason, we investigate \emph{$\alpha$-warped geometries} \cite{Chen1999} of Segre--Veronese manifolds. 
These geometries are described in detail in \cref{sec_warping_geometry}; essentially, they scale, by a factor of $\alpha > 0$, the length of all tangent directions that do not infinitesimally change the Euclidean norm.
The choice $\alpha = 1$ corresponds to the Euclidean geometry.
\Cref{fig_alphawarping} illustrates \new{several geodesic segments, specified either by an initial point and a tangent vector as in the left panel, or by two ``endpoints'' as in the right panel; that is, these can be interpreted as the geodesics captured by, respectively, the exponential and logarithmic maps.}

\begin{figure}[tb]
	\centering
	\begin{subfigure}{17.3em}
	\includegraphics[width=\textwidth]{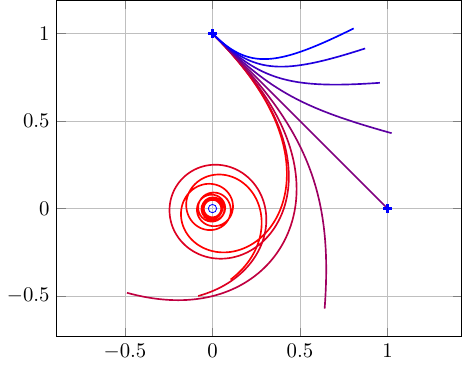}
	\caption{
		\new{Geodesics starting at $p=(0,1)$ in the tangent direction $\dot{p}=(1,-1)$ in several $\alpha$-warped geometries. All geodesic segments have the same length $\sqrt{2}$ in their $\alpha$-warped geometry, while their Euclidean lengths are seen to be very different. In the $\alpha$-warped geometries with $\alpha < 1$, the geodesics will first spiral towards the circle about the origin of radius $\alpha$ and then spiral out again from it. This will be evident from \cref{prop_geodesics,cor_exponential_map}.}%
	} \label{fig_alphawarping_left}
	\end{subfigure}
	\hfill
	\begin{subfigure}{17.3em}
	 	\includegraphics[width=\textwidth]{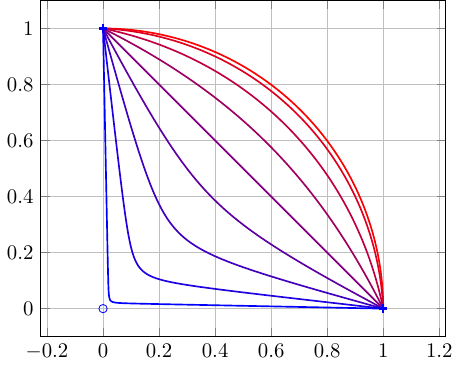}
		\caption{\new{Minimizing geodesic segments between $p=(0,1)$ and $q=(1,0)$ in several $\alpha$-warped geometries. As $\alpha \to 0$, the geodesic segment will approach the unit circle arc that connects $p$ and $q$. As $\alpha\to\infty$, the geodesic segment will converge to the piecewise linear path formed by the line segments from $(0,1)$ to $(1,0)$ via the node $(0,0)$. This will be evident from \cref{thm_logarithm_preSV,thm_logarithm_SV,prop_limiting_curve}}. \phantom{We need a bit more text to get to the next line.}}
		\label{fig_alphawarping_right}
	\end{subfigure}
	\vspace{-10pt}
	\caption{\new{Geodesics in the $\alpha$-warped geometry of the punctured plane $\Var{S}^{(1)}$ for $\alpha = 0.05$ (the red paths curving towards the origin, i.e., acceleration vectors pointing toward the origin), $\frac{1}{4}$, $\frac{1}{2}$, $\frac{3}{4}$, $1$ (the purple straight lines between $(0,1)$ and $(1,0)$), $\frac{5}{4}$, $\frac{3}{2}$, $\frac{7}{4}$, and $1.95$ (the blue paths curving away from the origin, i.e., acceleration vectors pointing away from $(0,0)$). Essentially, the left panel shows evaluations of the exponential map, and the right panel shows evaluations of the logarithmic map.}}\label{fig_alphawarping}
	\vspace{-10pt}
\end{figure}

\subsection{Outline}
The next section introduces background material from Riemannian geometry and describes $\alpha$-warped geometries of $\RR_*^n$. General $\alpha$-warped Segre--Veronese manifolds are obtained as Riemannian embedded submanifolds thereof. We also give an alternative characterization of $\alpha$-warped Segre--Veronese manifolds, revealing them to be covered by a warped product of the positive real numbers and a product of spheres, which we call a pre-Segre--Veronese manifold.
This characterization is exploited in the subsequent sections: the exponential maps of (pre-)Segre--Veronese manifolds are computed in \cref{sec_exp} and their logarithmic maps are computed in \cref{sec_log}.
We present the sectional curvatures of Segre--Veronese manifolds in \cref{sec_sectional_curvatures}.
Thereafter, in \cref{sec_applications}, we detail an example application of geodesics for computing the Fr\'echet mean.

\subsection*{Notation}
For convenience, scalars will be typeset in a lower case, vectors in a bold lower case, tensors in a calligraphic upper case, and manifolds in another calligraphic upper case.
We will emphasize that a vector lives in a tangent space by adding an upper dot, e.g., $\dot{\vect{v}}$.

Unless otherwise specified, $(\RR^n, \langle\cdot,\cdot\rangle)$ will be abbreviated to $\RR^n$; it is the inner product space of $n$-dimensional real vectors equipped with the standard Euclidean inner product.
If $\Var{M} \subset \RR^{n}$ is any subset, then 
\(
	\Sp(\Var{M}) := \set{ \vect{u} \in \Var{M} \mid \norm{\vect{u}}=1 } 
\)
is the space of vectors of $\Var{M}$ of unit Euclidean norm. We use the shorthand $\Sp^{n-1} := \Sp(\RR^n)$ for the unit sphere in $\RR^n$.
The angular distance from $\vect{u}$ to $\vect{v}$, i.e., the arclength of a shortest great circle segment connecting them, is $\sphericalangle(\vect{u}, \vect{v}) = \cos^{-1}(\langle \vect{v}, \vect{u} \rangle) \in [0, \pi]$.

Let $\Var{M}$ be a manifold.
Its tangent space at $x\in\Var{M}$ is denoted by $\Tang_x \Var{M}$.
The exponential map of $\Var{M}$ at $x$ is $\exp_x \colon \Omega_x \to \Var{M}$, where $\Omega_x$ is a suitable open neighborhood of $0 \in \Tang_x \Var{M}$.
The logarithmic map of $\Var{M}$ at $x$ is $\log_x \colon \Var{M}_x \to \Tang_x \Var{M}$, where $\Var{M}_x$ is a suitable open neighborhood of $x \in \Var{M}$.
For a \new{differentiable} curve $\gamma \colon \RR \to \Var{M}$, we denote the derivative to its parameter (see \cite[Section 3]{Lee13} for a definition) using a prime, i.e., 
\[
\gamma'(t) := \frac{\mathrm{d}}{\mathrm{d} t} \gamma (t) \quad\in \Tang_{x(t)} \Var{M}.
\]

\new{The exponential and logarithmic maps of a sphere $\mathbb{S}^{n-1}$ are, respectively,
\begin{align}
\label{eqn_exp_sphere}
 \exp{}^\mathbb{S}_{\vect{u}}(\dot{\vect{u}}) &= \vect{u} \cos(\|\dot{\vect{u}}\|) + \frac{\dot{\vect{u}}}{\|\dot{\vect{u}}\|} \sin(\|\dot{\vect{u}}\|), \quad\text{and}\\
 \label{eqn_log_sphere}
 \log{}^\mathbb{S}_{\vect{u}}(\vect{v}) &= (\vect{v} - \innerproduct{\vect{v}}{\vect{u}} \vect{u}) \frac{\sphericalangle(\vect{u}, \vect{v})}{\sin(\sphericalangle(\vect{u}, \vect{v}))};
\end{align}
see \cite[Examples 10.16 and 10.21]{Boumal2023}. The superscript is dropped from $\mathbb{S}^{n-1}$ in these maps as the dimension is also clear from the subscript $\vect{u}\in\mathbb{S}^{n-1}$.} 
\new{Note that we have $\norm{\log{}^\mathbb{S}_{\vect{u}}(\vect{v})}=\sphericalangle(\vect{u},\vect{v})$.}

\section{\texorpdfstring{$\alpha$}{alpha}-warped geometries for Segre--Veronese manifolds} \label{sec_warping_geometry}

The next subsection recalls preliminary results from Riemannian geometry.
In particular, the warped product of manifolds will be described.
Then, in \cref{sec_sub_SV}, the $\vect{k}$th Segre--Veronese manifold is formally defined and presented as a Riemannian embedded submanifold of a warped geometry of the space of tensors.
\Cref{sec_pre_SV} shows how Segre--Veronese manifolds can be realized as a normal Riemannian covering of a simpler warped product involving only spheres and the positive reals.

\subsection{Warped geometries of \texorpdfstring{$\RR^n_*$}{Rn*}}
\label{sub:warped_geometries}

Recall that a smooth embedded manifold $\Var{M}\subset\RR^N$ equipped with a \emph{Riemannian metric}---an inner product $g_x \colon \Tang_x \Var{M} \times \Tang_x \Var{M} \to \RR$ on its \emph{tangent spaces} $\Tang_x \Var{M}$ that satisfies some technical regularity conditions, see, e.g., \cite{Lee1997}---induces an intrinsic metric structure on $\Var{M}$ as follows.
A Riemannian metric induces a norm $\norm{\dot{\vect{x}}}_x := \sqrt{g_x(\dot{\vect{x}}, \dot{\vect{x}})}$ for all $\dot{\vect{x}}\in\Tang_x\Var{M}$, which can be used to measure the \emph{length} $\ell$ of piecewise smooth curves $\gamma \colon [0,1] \to \Var{M}$:
\begin{align*}
	\ell(\gamma) := \int_{0}^1 \norm{ \gamma'(t) }_{\gamma(t)} \dd{t}.
\end{align*}
Letting the set of all piecewise smooth curves be our length structure $L$, the manifold $\Var{M}$ becomes a \emph{length space} \cite{BBI2001}.
The \emph{intrinsic metric} is then
\begin{align*}
	\dist_\Var{M}(x, y) := \inf_{\substack{\gamma\in L,\\\gamma(0)=x,\, \gamma(1)=y}} \ell(\gamma).
\end{align*}
A diffeomorphism (a smooth bijection with smooth inverse) between manifolds that preserves the intrinsic metric is called an \emph{isometry}.

A piecewise smooth curve $\gamma \in L$ connecting $x = \gamma(0)$ and $y = \gamma(1)$ such that the length of $\gamma$ equals the distance $\dist_\Var{M}(x,y)$ is called a \emph{minimizing geodesic} \cite{Lee1997}.
That is, a minimizing geodesic is a shortest path in $L$ between $x$ and $y$ contained in $\Var{M}$.
More generally, a \emph{geodesic} is any piecewise smooth curve that satisfies the first-order optimality condition, i.e., is a critical point, of the variational problem $\min_{\gamma\in L} \ell(\gamma)$ subject to appropriate constraints (such as fixing two endpoints).
Such curves are always locally length-minimizing \cite[Theorem 6.12]{Lee1997} in the sense that they are minimizing geodesics on a segment $\gamma|_{[0,t)}$ for some $t > 0$.
Note that isometries map geodesics to geodesics.

For each tangent vector $\dot{\vect{v}} \in \Tang_x \Var{M}$, there is a unique constant-speed geodesic satisfying $\gamma(0) = x$, $\gamma'(0) = \dot{\vect{v}}$ \cite[Theorem 4.10]{Lee1997}.
\new{We can build} a map from (an \new{open subset} of) $\Tang_x \Var{M}$ to $\Var{M}$ via $\dot{\vect{v}} \mapsto \gamma(1)$;
this is the \emph{exponential map}.
It is a diffeomorphism between a neighborhood of $0 \in \Tang_x \Var{M}$ and a neighborhood of $x \in \Var{M}$~\cite[Lemma 5.10]{Lee1997}.
Its inverse is called the \emph{logarithmic map}.

We also mention two important concepts in Riemannian geometry: The \emph{Riemann \new{curvature} tensor} and the \emph{sectional curvature}.
The Riemann \new{curvature} tensor on a Riemannian manifold $\Var{M}$ is a multilinear map $R \colon \Tang_p \Var{M} \times \Tang_p \Var{M} \times \Tang_p \Var{M} \to \Tang_p \Var{M}$ that measures the change in a tangent vector as it is parallel transported around a small loop.
The sectional curvature is a map $K \colon \Tang_p \Var{M} \times \Tang_p \Var{M} \to \mathbb{R}$ that measures the angular defect of small triangles around $p$.
See \cite[Chapter 7]{Lee1997} for precise definitions.

The punctured space $\RR^n_*$ is an open submanifold of $\RR^n$; hence, it is a smooth manifold.
It can be equipped with the standard Euclidean metric $g$ at $\vect{a}\in\RR^n_*$, so that
\(
 g_{\vect{a}}( \dot{\vect{x}}, \dot{\vect{y}} ) = \langle \dot{\vect{x}}, \dot{\vect{y}} \rangle = \dot{\vect{x}}^{\mathsf{T}} \dot{\vect{y}},
\)
where $\dot{\vect{x}}$, $\dot{\vect{y}} \in \Tang_\vect{a} \RR^n_* = \RR^n$ and $\langle\cdot,\cdot\rangle$ is the standard Euclidean inner product. This leads to the Euclidean geometry of $(\RR^n_*, g)$. 
Other meaningful geometries of $\RR^n_*$ exist. For example, $\RR^n_*$ is diffeomorphic to the product of the positive real line and an $(n-1)$-dimensional sphere:
\(
 \RR^n_* \simeq \RR_{>0} \times \Sp^{n-1}.
\)
Under this identification, the above standard Euclidean metric can be expressed as
\begin{align*}
	g_{\lambda \vect{u}}( \dot{x} \vect{u} + \lambda \dot{\vect{v}}, \dot{y} \vect{u} + \lambda \dot{\vect{w}} ) 
	= \dot{x} \dot{y} + \lambda^2 \langle \dot{\vect{v}}, \dot{\vect{w}} \rangle,
\end{align*}
where $\dot{x}$, $\dot{y} \in \RR = \Tang_\lambda \RR_{>0}$ and $\dot{\vect{v}}$, $\dot{\vect{w}} \in \vect{u}^\perp = \Tang_{\vect{u}} \Sp^{n-1}$.
Herein, the inner products on the right are the inner products inherited from the ambient Euclidean spaces $\RR$ and $\RR^n$, respectively.
Note in particular that the metric that $\RR^n_*$ inherits from $\RR^n$ is \emph{not} equal to the product metric, $\dot{x} \dot{y} + \innerproduct{\dot{\vect{v}}}{\dot{\vect{w}}}$, of $\RR_{>0}$ and $\Sp^{n-1}$.
Instead, it is a so-called \emph{warped product} \cite[Chapter 7]{ONeill1983} with the identity map as \emph{warping function}.
Recall from \cite[Chapter 7]{ONeill1983} that if we are given Riemannian manifolds $(\Var{M},g)$ and $(\Var{N}, h)$, then we can construct their warped product with warping function $f \colon \Var{M} \to \RR$ as \(
\Var{M} \times_f \Var{N} := (\Var{M}\times\Var{N}, g \times_f h),
\)
where
\begin{align*}
	(g \times_f h)_{(m,n)}( (\dot{\vect m}_1, \dot{\vect n}_1), (\dot{\vect m}_2, \dot{\vect n}_2) ) = g_m(\dot{\vect m}_1, \dot{\vect m}_2) + f^2(m) \cdot h_n(\dot{\vect n}_1, \dot{\vect n}_2).
\end{align*}

Rather than considering only the Euclidean geometry of $\RR^n_*$, we also consider its \emph{$\alpha$-warped geometries} in this paper; that is,
\begin{align}\label{eqn_ambient_geometry}
	\Var{R}_\alpha = (\RR^n_*, \varsigma^\alpha) := \RR_{>0} \times_{\alpha \mathrm{Id}} \Sp^{n-1} \quad\text{with } \alpha > 0,
\end{align}
whose warped product metric satisfies 
\begin{align}
	\varsigma^{\alpha}_{\lambda\vect{u}} (\dot{x} \vect{u} + \lambda \dot{\vect{v}}, \dot{y} \vect{u} + \lambda \dot{\vect{w}} )
	= \dot{x} \dot{y} + (\alpha \lambda)^2 \langle \dot{\vect{v}}, \dot{\vect{w}} \rangle.
\end{align}
If $\alpha = 1$, then the $1$-warped geometry is the Euclidean geometry of $\RR^n_*$, with straight lines as geodesics between points.
If $\alpha > 1$, then distances in the spherical directions will be elongated, implying that geodesics between points will tend to be forced toward the origin.
On the other hand, if $0 < \alpha < 1$, then spherical movements are easier, implying that geodesics between points will bend away from the origin.
This is visualized in \cref{fig_alphawarping_right}.

\subsection{The \texorpdfstring{$\alpha$}{alpha}-warped Segre--Veronese submanifold} \label{sec_sub_SV}

The punctured space $\RR^n_*$ is the most basic example (with $\vect{k}=(1)$) of the $\vect{k}$th Segre--Veronese manifold
\begin{align}
	\Var{S}^\vect{k}_\pm := \set{ \lambda \vect{x}_1^{\otimes k_1} \otimes \cdots \otimes \vect{x}_d^{\otimes k_d} \mid \lambda\in\RR_*,\; \vect{x}_i\in\Sp^{n_i-1},\, i=1,\ldots,d },
\end{align}
where $\vect{x}_i^{\otimes k_i}$ is the $k_i$-fold \emph{tensor product} of $\vect{x}_i$ with itself, i.e., $\vect{x}_i^{\otimes k_i} := \vect{x}_i \otimes\cdots\otimes \vect{x}_i$. 
Recall that the tensor product of vectors in $\RR^{n_j}$, $j = 1$, \dots, $d$, can be embedded naturally into the space of $n_1\times\cdots\times n_d$ arrays, as follows:
\begin{align*}
	(\vect{x}_1 \otimes\cdots\otimes \vect{x}_d)(i_1,\ldots,i_d) := \vect{x}_1(i_1) \cdots \vect{x}_d(i_d), \qquad 1 \le i_j \le n_j,\, j=1,\ldots,d,
\end{align*}
where $\vect{x}_j \in \RR^{n_j}$ and $\vect{x}_j(i_j)$ denotes the $i_j$th coefficient of $\vect{x}_j$ with respect to a basis of $\RR^{n_j}$; see, e.g., \cite{Lim2021}.
That is, the coordinates of a tensor product are obtained by multiplying the coordinates of all vectors in all possible ways.

Let us also define the \enquote{positive} Segre--Veronese manifold
\begin{align}
	\Var{S}^\vect{k} := \set{ \lambda \vect{u}_1^{\otimes k_1} \otimes\cdots\otimes \vect{u}_d^{\otimes k_d} \mid \lambda \in \RR_{>0},\; \vect{u}_i \in \Sp^{n_i-1},\, i=1,\ldots,d },
\end{align}
which is a smooth embedded submanifold (see e.g. \cite[Section 4.3]{Landsberg2012} and \cite[Chapter 5]{Lee13}) that is diffeomorphic to
\(
\RR_{>0} \times \Sp(\Var{S}^\vect{k}).
\)
We can thus equip it with the $\alpha$-warped metric that it inherits as a submanifold of $\Var{R}_\alpha = \RR_{>0} \times \Sp^{N-1}$ from \cref{eqn_ambient_geometry}, where $N = n_1^{k_1}\cdots n_d^{k_d}$. The resulting Riemannian submanifold will be denoted by
\begin{align}
	\Var{S}_\alpha^\vect{k} := \RR_{>0} \times_{\alpha \mathrm{Id}} \Sp(\Var{S}^\vect{k}) = (\Var{S}^\vect{k}, \varsigma^\alpha) \subset \Var{R}_\alpha,
\end{align}
and we will refer to it as the \emph{$\alpha$-warped Segre--Veronese manifold}.

\begin{remark}
If all $k_i$'s are even, then $\Var{S}_\pm^\vect{k}$ consists of two \emph{connected components}, namely
the positive Segre--Veronese manifold $\Var{S}^\vect{k}$ and its mirror image along the origin, namely the \enquote{negative} Segre--Veronese manifold $-\Var{S}^{\vect{k}}$. 
This means that $p\in\Var{S}^\vect{k}$ and $q\in-\Var{S}^\vect{k}$ are not connected by any curve in $\Var{S}^\vect{k}_\pm$. In particular, there will be no geodesics connecting them.
Nevertheless, the two connected components are \emph{isometric}, as $\varsigma_{(-\lambda)\vect{u}}^\alpha = \varsigma_{\lambda \vect{u}}^\alpha$ (herein, the metric should be interpreted as being defined on $\RR_{\ne0}$), and we have $-\Var{S}^\vect{k} = \RR_{<0} \times_{\alpha\mathrm{Id}} \Sp^{N-1}$. As these two connected components are isometric manifolds, their geometry is the same.

On the other hand, if at least one of the $k_i$'s is odd, then it is easy to show that $\Var{S}_\pm^\vect{k} = \Var{S}^\vect{k} = -\Var{S}^{k}$. That is, the \enquote{regular}, \enquote{positive}, and \enquote{negative} flavors of the Segre--Veronese manifold all coincide.

Therefore, we will henceforth study the component $\Var{S}^\vect{k}$ without loss of generality.
\end{remark}

\subsection{The \texorpdfstring{$\alpha$}{alpha}-warped pre-Segre--Veronese manifold} \label{sec_pre_SV}

We study the geometry of $\Var{S}_\alpha^\vect{k}$ through the lens of the tensor product $\otimes$. By definition, $\Var{S}^\vect{k}$ is the image of
\begin{align*}
	\otimes \colon \RR_{>0} \times \Sp^{n_1-1} \times\cdots\times \Sp^{n_d-1} &\longrightarrow \RR^{n_1^{k_1}\times\cdots\times n_d^{k_d}}_*,\nonumber\\
	(\lambda, \vect{u}_1, \ldots, \vect{u}_d) &\longmapsto \lambda \vect{u}_1^{\otimes k_1} \otimes \cdots \otimes \vect{u}_d^{\otimes k_d}.
\end{align*}
As the domain is a product manifold, its tangent space at $(\lambda, \vect{u}_1, \ldots, \vect{u}_d)$ is 
\begin{align*}
	\Tang_\lambda \RR_{>0} \times \Tang_{\vect{u}_1} \Sp^{n_1-1} \times\cdots\times \Tang_{\vect{u}_d} \Sp^{n_d-1} = \RR \times \vect{u}_1^\perp \times\cdots\times \vect{u}_d^\perp.
\end{align*}
We can equip it with the $\alpha$-warped product metric
\begin{align}\label{eqn_warped_metric_preSV}
	g_{(\lambda,\vect{u}_1,\ldots,\vect{u}_d)}
	\bigl( 
	(\dot{x}, \dot{\vect{u}}_1, \ldots, \dot{\vect{u}}_d), 
	(\dot{y}, \dot{\vect{v}}_1, \ldots, \dot{\vect{v}}_d)
	\bigr)
	:= 
	\dot{x} \dot{y} + (\alpha\lambda)^2 \sum_{i=1}^d k_i \langle \dot{\vect{u}}_i, \dot{\vect{v}}_i \rangle.
\end{align}
The resulting Riemannian manifold will be called the \emph{$\alpha$-warped pre-Segre--Veronese manifold}, and it will be denoted by
\begin{equation}\label{eqn_def_preSV}
	\Var{P}_\alpha^\vect{k}
	= (\RR_{>0}, \langle\cdot,\cdot\rangle) \times_{\alpha \mathrm{Id}} \underbrace{( (\Sp^{n_1-1}, k_1\langle\cdot,\cdot\rangle) \times\cdots\times (\Sp^{n_d-1}, k_d \langle\cdot,\cdot\rangle) )}_{\Sp^{\vect{n},\vect{k}}}.
\end{equation}
Note that the standard Euclidean inner product on the $i$th sphere is weighted by a factor $k_i$. The geometry of such a sphere $(\Sp^{n_i-1}, k_i\langle\cdot,\cdot\rangle)$ is equivalent to the usual Euclidean geometry of the sphere $(\Sp^{n_i-1}_{\sqrt{k_i}}, \langle\cdot,\cdot\rangle)$ of radius $\sqrt{k_i}$.

The key result is that the tensor product takes the pre-Segre--Veronese manifold $\Var{P}_\alpha^\vect{k}$ in a \new{standard} way to the Segre--Veronese manifold $\Var{S}_\alpha^\vect{k}$.
Specifically, it is a \emph{normal Riemannian covering map} \cite{ONeill1983}, which means that (i) every point $\tensor{T} \in \Var{S}_\alpha^{\vect{k}}$ has the same number, \new{$P$}, of preimages in $\Var{P}_\alpha^\vect{k}$, (ii) every point $\tensor{T} \in \Var{S}_\alpha^{\vect{k}}$ has a neighborhood $U_{\tensor{T}}$ such that its preimage $\otimes^{-1}(U_{\tensor{T}})$ is a disjoint union of \new{$P$} domains, and (iii) on each of those domains, $\otimes$ is an isometry.
Property (iii) is called \emph{local isometry}.

A diffeomorphism $\iota \colon \Var{P}_\alpha^\vect{k} \to \Var{P}_\alpha^\vect{k}$ that satisfies $\otimes \circ \iota = \otimes$ is called a \emph{deck transformation}.
\new{Intuitively}, it shuffles the preimages of the covering map.
A deck transform that is also an isometry is called an \emph{isometric deck transform}.

\begin{lemma} \label{lem_riemannian_covering}
	The map
	\begin{align*}
  		\otimes \colon \Var{P}_\alpha^\vect{k} \to \Var{S}_\alpha^\vect{k}, \quad (\lambda, \vect{u}_1, \ldots, \vect{u}_d) \mapsto \lambda \vect{u}_1^{\otimes k_1}\otimes\cdots\otimes\vect{u}_d^{\otimes k_d}
 	\end{align*}
	is a normal Riemannian covering where the isometric deck transforms are of the form $\imath_\sigma(\lambda, \vect{u}_1,\ldots,\vect{u}_d) = (\lambda,\sigma_1\vect{u}_1,\ldots,\sigma_d \vect{u}_d)$ with $\sigma_i\in\set{-1, 1}$ and $\sigma_1^{k_1}\cdots\sigma_d^{k_d} = 1$.
\end{lemma}

An elementary but crucial fact about normal Riemannian coverings concerns the lengths of curves in the covered manifold.
The next result enables the computation of geodesics in $\Var{S}_\alpha^\vect{k}$ by studying the geodesics of the warped product $\Var{P}_\alpha^\vect{k}$.

\begin{lemma}\label{lem_basic_covering_property}
	Let $\phi \colon \Var{M} \to \Var{N}$ be a normal Riemannian covering,
	$\gamma \colon [0, 1] \to \Var{N}$ a smooth curve, and
	$p \in \Var{M}$ any point such that $\phi(p)=\gamma(0).$
	Then, there exists a unique smooth lift $\widetilde{\gamma} \colon [0, 1] \to \Var{M}$ with $\widetilde{\gamma}(0)=p$ and $\phi \circ \widetilde{\gamma} = \gamma$ such that $\ell(\gamma) = \ell(\widetilde{\gamma})$.

	In particular, a minimizing geodesic segment in $\Var{N}$ lifts to a minimizing geodesic segment in $\Var{M}$.
\end{lemma}

These lemmas are proved in \cref{sec:proof_of_lemmas}.

By \cref{lem_basic_covering_property}, if we want to find a minimizing geodesic \new{that connects} $\tensor{P}$ \new{and} $\tensor{Q} \in \Var{S}_\alpha^{\vect{k}}$, it suffices to enumerate the geodesics \new{connecting their} preimages \new{in the $\alpha$-warped pre-Segre--Veronese manifold} and pick the shortest one. This is essentially what our proposed matchmaking algorithm will do in \cref{prop_matching_algorithm}.

\new{The geodesics of $\alpha$-warped product manifolds in general have been studied by Chen \cite{Chen1999}. In particular, \cite[Lemma 4.3]{Chen1999} states that geodesics in the warped product $\Var{M} \times_f \Var{N}$ have the property that their projection onto the second factor $\Var{N}$ also yields a geodesic, albeit not parameterized to be of constant speed. We can effectively leverage this result to parameterize geodesic segments passing through a given point in the pre-Segre--Veronese manifold.}

\new{\begin{lemma}\label{lem_basic_geodesic_form}
\begin{subequations}\label{eqn_paths}
Let $p=(\lambda,\vect{u}_1,\dots,\vect{u}_d) \in \Var{P}_{\alpha}^\vect{k}$ be a point in the $\alpha$-warped pre-Segre--Veronese manifold. If $\sigma \subset \Var{P}_\alpha^\vect{k}$ is a geodesic passing through $p$, then there exist $\eta \in \RR_{>0}$, $B \in \RR$, and tangent vectors $\dot{\vect{v}}_i \in \Tang_{\vect{u}_i} \mathbb{S}^{n_i - 1}$ for $i=1,\dots,d$ such that
\begin{equation}\label{eqn_geodesic_segment}
	\sigma(s) = \left(\mu(s),\vect{v}_1(s),\dots,\vect{v}_d(s)\right), \quad s\in[0,\eta),
\end{equation}
parameterizes a geodesic segment with $\sigma(0) = p$ and either
\[
\sigma(s) = (\lambda + s, \vect{u}_1,\dots,\vect{u}_d)
\]
if all $\dot{\vect{v}}_i = 0$, or
\begin{enumerate}
 \item $\sigma(s)$ has unit spherical speed, i.e.,
\begin{align}\label{eqn_unit_spherical_speed}
	\norm{( \vect{v}_1'(s) , \dots, \vect{v}_d'(s))}_{(\vect{v}_1(s),\ldots,\vect{v}_d(s))} =1;
\end{align}
 \item the real curve is
 \begin{align}\label{eqn_musol}
	\mu(s) = \lambda \frac{\cos(B)}{\cos(\alpha s + B )};
 \end{align}
 \item the spherical curves are great circles, i.e.,
\begin{equation}\label{eqn_solsphere}
	\vect{v}_i(s)
	= \exp^\mathbb{S}_{\vect{u}_i} ( s \dot{\vect{v}}_i ),
	\quad i=1,\ldots,d,
\end{equation}
where $\exp^\mathbb{S}_{\vect{u}_i}$ is the exponential map of the sphere $\mathbb{S}^{n_i-1}$ as in \cref{eqn_exp_sphere}.
\end{enumerate}
\end{subequations}
\end{lemma}}
\begin{proof}
	\new{Since $\Var{P}_\alpha^\vect{k} = \RR_{>0} \times_{\alpha\mathrm{Id}} \mathbb{S}^{\vect{n},\vect{k}}$ is a warped product, Chen's result \cite[Lemma 4.3]{Chen1999} implies that a geodesic in $\Var{P}_\alpha^\vect{k}$ projects to a geodesic of $\mathbb{S}^{\vect{n},\vect{k}}$. The geometry of $\mathbb{S}^{\vect{n},\vect{k}}$ is the standard product geometry of the Euclidean spheres $\mathbb{S}^{n_i-1}_{\sqrt{k_i}}$ of radius $\sqrt{k_i}$. Hence, the geodesics of $\mathbb{S}^{\vect{n},\vect{k}}$ are the products of the geodesics in the individual spheres $\mathbb{S}^{n_i-1}_{\sqrt{k_i}}$, which are all great circles.
	This shows that there exist tangent vectors $\dot{\vect{v}}_i$ so that all $\vect{v}_i(s)$'s are of the form \cref{eqn_solsphere}.}

	The coefficient $\mu \colon \mathbb{R} \to \mathbb{R}$ \new{in \cref{eqn_geodesic_segment}} is determined by solving the \textit{Euler--Lagrange equations} of the length functional, which in the $\alpha$-warped geometry becomes
	\begin{align} \label{eqn_length_functional}
		D
		\nonumber &= \int_0^{1} \sqrt{\langle \sigma'(s),\sigma'(s) \rangle_{\sigma(s)}}\dd{s} \\
 		&=\int_0^{1} \sqrt{ \mu'(s)^2+ \alpha^2 \mu(s)^2 \norm{( \vect{v}_1'(s) , \dots, \vect{v}_d'(s))}_{(\vect{v}_1(s),\ldots,\vect{v}_d(s))}^2 }\ \dd{s}.
	\end{align}

	\new{If all tangent vectors $\dot{\vect{v}}_i = 0$, then the great circles $\vect{v}_i(s)=\vect{v}_i$ are constant for all $s$. Therefore, the above length functional reduces to $D = \int_{0}^1 \sqrt{\mu'(s)^2} \dd{s}$, which is the standard length functional of the real line. The geodesics in Euclidean space are straight lines, so $\mu(s)$ must be a parameterization of a straight line. This proves the form of $\sigma$ in this special case.}

	\new{In the remainder of the proof, we treat the case where some $\dot{\vect{v}}_i \ne 0$. Then, we can assume
	$\sigma$ has been parameterized so that \cref{eqn_unit_spherical_speed} holds. The critical points of the Euler--Lagrange equation are the solutions of $\frac{\partial L}{\partial \mu} = \frac{\partial}{\partial s} \frac{\partial L}{\partial \mu'}$, where $L(s, \mu, \mu') = \sqrt{\mu'(s)^2 + \alpha^2 \mu(s)^2}$.}
	This ultimately leads to the differential equation
	\(
	\alpha^2 \mu(s)^2 + 2 \mu'(s)^2 - \mu(s) \mu''(s) = 0
	\)
	whose solution is of the form
	\(
		\mu(s) = \frac{A}{\cos(\alpha s + B )}.
	\)
	Solving for the constant $A$ by exploiting the starting conditions \new{$\sigma(0) = p$, so $\mu(0)=\lambda$}, we arrive at $A = \lambda \cos(B)$. This proves the form \cref{eqn_paths} holds.
\end{proof}

\new{Note that the characterization of geodesics in \cref{lem_basic_geodesic_form} is not complete in the sense there are unspecified quantities ($\eta, B, \dot{\vect{v}}_1, \dots, \dot{\vect{v}}_d$). This is as expected, because geodesics are solutions to the Euler--Lagrange differential equation. Unique solutions to this second-order differential equation require two interpolation conditions, such as passing through a point and its tangent vector (yielding the exponential map), or passing through two points (yielding the logarithmic map).}

\section{The exponential map}\label{sec_exp}

Before presenting the main original contribution in \cref{sec_log}, we briefly investigate the exponential maps of (pre-)Segre--Veronese manifolds.
The following result is a generalization of~\cite[Theorem 1]{SVV2022} and its proof, which dealt only with the usual Euclidean geometry (i.e., $\alpha=1$) of the Segre manifold (i.e., $\vect{k}=(1,\ldots,1)$).

\begin{proposition}[Exponential map of $\Var{P}_\alpha^\vect{k}$] \label{prop_geodesics}
	Let
	\begin{align*}
		p = (\lambda,\vect{u}_1,\dots,\vect{u}_d) \in \Var{P}_\alpha^\vect{k} \quad\text{and}\quad
		\dot{p} = (\dot \lambda,\dot{\vect{u}}_1,\dots,\dot{\vect{u}}_d) \in \Tang_p \Var{P}_\alpha^\vect{k}.
	\end{align*}
	Then, \new{the exponential map satisfies}
	\begin{align*}
		\exp_p(\dot{p}) = \left( \sqrt{(\lambda + \dot{\lambda})^{2} + (\lambda \alpha N)^{2}},
		\new{\exp{}^\mathbb{S}_{\vect{u}_1}(a \dot{\vect{u}}_1),
		\dots,
		\exp{}^\mathbb{S}_{\vect{u}_d}(a \dot{\vect{u}}_d)}
		\right),
	\end{align*}%
	where
	\begin{align*}
		N = \sqrt{\sum_{i=1}^d k_i \norm{\dot{\vect{u}}_i}^2} \quad\text{and}\quad
		a = \frac{1}{\alpha N} \left( \frac{\pi}{2} - \arctan\left( \frac{\lambda + \dot{\lambda}}{\lambda \alpha N} \right) \right).
	\end{align*}
	If $N = 0$ and $-\lambda < \dot{\lambda}$, the geodesics are straight lines:
	\begin{align*}
		\exp_{p}(\dot{p}) = (\lambda+\dot{\lambda}, \vect{u}_1, \ldots, \vect{u}_d).
	\end{align*}
\end{proposition}

\begin{proof}
	\new{Recall that the exponential map evaluated at $\dot{p}$ is equivalent to evaluating the constant-speed} geodesic $\gamma(t) = (\lambda(t),\vect{u}_1(t),\dots,\vect{u}_d(t))$ in $\Var{P}_\alpha^\vect{k}$ with starting conditions $\gamma(0)=p$ and $\gamma'(0) = \dot{p}$ \new{at $t=1$}.

	If $N=0$, \new{the form of the geodesic, and hence exponential map, follows immediately from \cref{lem_basic_geodesic_form}.}
	So we only need to consider the case $N > 0$.

	We can apply the change of variable $t \mapsto s(t)$ with $s(0) = 0$ and $s'(0) > 0$ such that $\gamma(t) = \sigma(s(t))$, \new{where $\sigma(s)$ is as in \cref{eqn_paths} from \cref{lem_basic_geodesic_form}.}
	\new{Using the initial condition $\gamma'(0)=\dot{p}$ and $\frac{\dd}{\dd t} \exp_{\vect{u}_i}^\mathbb{S}(s(t) \dot{\vect{v}}_i) = \dot{\vect{v}}_i s'(t)$, we determine that
	\[
 	\dot{\vect{u}}_i = \vect{v}_i'(0) = s'(0) \dot{\vect{v}}_i,
	\]
	where $\dot{\vect{v}}_i$ are the tangent vectors from \cref{lem_basic_geodesic_form}.
	Moreover,} \(\norm{\vect{v}_i'(0)}_{\vect{v}_i(0)} s'(0) = \sqrt{k_i} \norm{\dot{\vect{u}}_i}\),
	$i = 1, \dots, d$, from which we find that $s'(0) = N$, where $N$ is as in the statement of the proposition.
	Solving for the constant $B$ \new{from \cref{lem_basic_geodesic_form}} by exploiting the starting condition \new{$\mu'(0) = \dot{\lambda}$}, we arrive at $B = \arctan(\dot{\lambda} / (\alpha \lambda N))$.
	
	\new{Let $A = \lambda \cos(B)$.} We have that
	\begin{align}
		\norm{\gamma'(t)}_{\gamma(t)} ={}&
		\abs{s'(t)} \sqrt{ {\mu}'\mleft(s(t)\mright)^2 + \alpha^2 \mu\mleft(s(t)\mright)^2 } \nonumber\\
		={}& \abs{s'(t)} \sqrt{\frac{\alpha^2 A^2\sin^2\mleft(\alpha s(t)+B\mright)}{\cos^4\mleft(\alpha s(t)+B\mright)}+\frac{\alpha^2 A^2}{\cos^2\mleft(\alpha s(t)+B\mright)}} \nonumber\\
		={}& \abs{s'(t)} \frac{\alpha A}{\cos^2(\alpha s(t) + B)}. \label{eqn_proof_step}
	\end{align}
	Since $\sigma(s)$ is injective, the parameterization $s(t)$ is a monotone increasing function.
	Hence, $s'(t) \geq 0$.
	By plugging $\norm{\gamma'(t)}_{\gamma(t)} = \norm{\dot{p}}_{p} = \sqrt{\dot{\lambda}^{2} + (\alpha \lambda N)^{2}}$ into \cref{eqn_proof_step}, we obtain the differential equation
	\begin{align}\label{eq:ivp}
	  s'(t) = \cos^2(\alpha s(t) + B) \frac{\sqrt{\dot{\lambda}^{2} + (\alpha \lambda N)^{2}}}{\alpha A}.
	\end{align}
	Combining this with the initial condition $s(0) = 0$, we have the solution
	\begin{equation} \label{eqn_reparameterization}
		s(t) = \alpha^{-1} \left( \frac{\pi}{2} - \arctan \left( \frac{\lambda + \dot{\lambda} t}{\alpha \lambda N t} \right) \right).
	\end{equation}
	Verifying that this solves \cref{eq:ivp} is a straightforward computation; see \cref{sub:Proof of eqn_reparameterization}.
	
	To obtain the exponential map, we now just have to evaluate $\gamma(1)$. For the first component of $\gamma(1)$ we find, after some computations (see \cref{sub:Proof of eq:gamma_of_1}), that
	\begin{align}\label{eq:gamma_of_1}
		\mu(s(1)) ={}& \sqrt{(\lambda + \dot{\lambda})^{2} + (\alpha \lambda N)^{2}}.
	\end{align}
	The other components $\vect{v}_i(s(1))$ are obtained immediately by plugging the expression of $s(t)$ from \cref{eqn_reparameterization} into \cref{eqn_solsphere}.
	This concludes the proof.
\end{proof}

From this result, it immediately follows that the exponential map is well defined almost everywhere.

\begin{corollary}\label{cor_domain_exp}
	The domain of the exponential map $\exp_p$ of $\Var{P}_\alpha^\vect{k}$ is the whole tangent space $\Tang_p \Var{P}_\alpha^\vect{k}$, except for the half-line $\set{ (\dot{\lambda},0,\ldots,0) \mid \dot{\lambda} \le -\lambda }$.
\end{corollary}

Having computed the geodesics in the pre-Segre--Veronese manifold, an application of \cref{lem_riemannian_covering,lem_basic_covering_property} enables us to push them forward with $\otimes$ to the Segre--Veronese manifold.
Note that because Riemannian coverings are \emph{local} isometries, we are only guaranteed that geodesics are mapped to geodesics, but not that minimizing geodesics are mapped to minimizing geodesics, i.e., the converse of \cref{lem_basic_covering_property} is not true.
In this way, we arrive at the geodesics and the following characterization of the exponential map of $\Var{S}_\alpha^\vect{k}$.

\begin{corollary}[The exponential map of $\Var{S}_\alpha^\vect{k}$] \label{cor_exponential_map}
	Let $\tensor{P} \in \Var{S}_\alpha^\vect{k}$ be a rank-$1$ tensor and let $p=(\lambda,\vect{u}_1,\ldots,\vect{u}_d) \in \otimes^{-1}(\tensor{P})$ be any of its representatives in $\Var{P}_\alpha^\vect{k}$.
	Let $\dot{\tensor{P}} \in \Tang_\tensor{P} \Var{S}_\alpha^\vect{k}$, which can be expressed uniquely as
	\begin{align*}
 		\dot{\tensor{P}} = \dot{\lambda}\tensor{U} + \lambda\mleft( \nu_{k_1}(\dot{\vect{u}}_1)\otimes\vect{u}_2^{\otimes k_2}\cdots\otimes\vect{u}_d^{\otimes k_d} + \cdots + \vect{u}_1^{\otimes k_1}\otimes\cdots\otimes\vect{u}_{d-1}^{\otimes k_{d-1}}\otimes\nu_{k_d}(\dot{\vect{u}}_d) \mright),
	\end{align*}
	where
	\begin{align}
		\tensor{U} &=\vect{u}_1^{\otimes k_1} \otimes \cdots \otimes \vect{u}_d^{\otimes k_d},%
		\label{eqn_U}\\
		\nu_{k_i}(\dot{\vect{u}}_i) &= (\dot{\vect{u}}_i \otimes\vect{u}_i^{\otimes (k_i - 1)}) + \cdots + (\vect{u}_i^{\otimes (k_i - 1)} \otimes \dot{\vect{u}}_i) \quad\text{for all } \dot{\vect{u}}_i \in \Tang_{\vect{u}_i} \Sp^{n}.%
		\label{eqn_Nu}
	\end{align}
	Let $\dot{p}=(\dot{\lambda},\dot{\vect{u}}_1,\ldots,\dot{\vect{u}}_d) \in \Tang_p \Var{P}_\alpha^\vect{k}$ be the corresponding lifted tangent vector.
	Then,
	\begin{align*}
 		\exp_\tensor{P}(\dot{\tensor{P}}) = \otimes ( \exp_{p}(\dot{p}) ),
	\end{align*}
	insofar as $\dot{p}$ is in the domain of $\exp_p$.
\end{corollary}

Another interesting consequence of \cref{prop_geodesics,cor_domain_exp} is that it entails the next result about the radius of the largest open ball in $\Tang_\tensor{P} \Var{S}_\alpha^\vect{k}$ for which $\exp$ is a diffeomorphism, i.e., the local injectivity radius of the Segre--Veronese manifold.

\begin{corollary}[Injectivity radius]
	The local injectivity radius of $\Var{P}_\alpha^\vect{k}$, and thus also of $\Var{S}_\alpha^\vect{k}$, at $p = (\lambda,\vect{u}_1,\ldots,\vect{u}_d)$ is $\lambda$.
	The (global) injectivity radius of $\Var{P}_\alpha^\vect{k}$ and $\Var{S}_\alpha^\vect{k}$, i.e., the infimum of the local injectivity radii over the manifold, is zero.
\end{corollary}

Note that the latter claim is also immediate from the fact that neither $\Var{P}_\alpha^\vect{k}$ nor $\Var{S}_\alpha^\vect{k}$ are complete metric spaces.
Indeed, the sequence $p_n = (n^{-1}, \vect{u}_1, \ldots, \vect{u}_d) \in \Var{P}_\alpha^\vect{k}$ tends to $(0,\vect{u}_1,\ldots,\vect{u}_d) \not\in \Var{P}_\alpha^\vect{k}$ as $n\to\infty$, and the image of $p_n$ in $\Var{S}_{\alpha}^{\vect{k}}$ tends to $0 \notin \Var{S}_{\alpha}^{\vect{k}}$.

Finally, we highlight that the exponential map of the pre-Segre--Veronese manifold $\Var{S}_\alpha^\vect{k}$ can be computed numerically very efficiently by straightforward implementation of the formulas in \cref{prop_geodesics}.
The cost is about $\mathcal{O}(\sum_{i=1}^d n_i)$ operations to compute $N$ and all $\vect{u}_i(t)$'s, and a constant number of operations for $\lambda(t)$ and $g(t)$. 

\section{The logarithmic map}\label{sec_log}

We now compute the logarithmic map of the $\alpha$-warped (pre-)Segre--Veronese manifold, significantly extending prior results of one of the authors, \new{Lars Swijsen}, on the Segre manifold ($\vect{k}=(1,\ldots,1)$) with the Euclidean metric ($\alpha=1$) as part of his PhD thesis~\cite{Swijsen2022}.

In the next subsection, we first determine a practical characterization of \emph{when} two points in the pre-Segre--Veronese manifold can be connected with a geodesic.
Then, in \cref{sec_sub_preSVlog}, we compute this minimizing geodesic if it exists.
Finally, we compute minimizing geodesics in the Segre--Veronese manifold between two rank-$1$ tensors in \cref{sec_sub_SVlog} by appropriately lifting them to the pre-Segre--Veronese manifold.
Additionally, an efficient matchmaking algorithm is presented to compute this lift, essentially solving a particular combinatorial problem in linear time.

\subsection{Compatibility}

Our main motivation for considering $\alpha$-warped geometries of the Segre--Veronese manifold is the observation in \cite[Chapter 6]{Swijsen2022} that not all pairs of points in the Segre manifold are connected by a minimizing geodesic.
That is, the Segre manifold is not geodesically connected in the standard Euclidean geometry.

The main result of this subsection, \cref{prop_compatibility}, will show that the geodesic connectedness of $\Var{P}_\alpha^\vect{k}$ can be characterized by the following property.

\begin{definition}[$\alpha$-compatibility] \label{def_compatibility}
Consider two points
\(
p = (\lambda, \vect{u}_1, \ldots, \vect{u}_d)
\) and \(
q = (\mu, \vect{v}_1, \ldots, \vect{v}_d)
\) 
in $\Var{P}_\alpha^\vect{k}$. Let
\begin{align*}
	M(p,q) := \dist_{\Sp^{\vect{n},\vect{k}}}( (\vect{u}_1,\ldots,\vect{u}_d), (\vect{v}_1,\ldots,\vect{v}_d) ) = \sqrt{\sum_{i=1}^d k_i \cdot \sphericalangle^2( \vect{u}_i, \vect{v}_i )}.
\end{align*}
Then, $p$ and $q$ are called \emph{$\alpha$-compatible} if
\(
M(p,q) < \alpha^{-1} \pi.
\)
\end{definition}

The dependency of $M(p,q)$ on the parameters $\vect{k} \ge 1$ will not be emphasized in the notation, instead being implicit in the fact that $p$, $q\in\Var{P}_\alpha^\vect{k}$.

\begin{proposition}\label{prop_compatibility}
	Two points $p$ and $q$ can be connected by a minimizing geodesic in $\Var{P}_\alpha^\vect{k}$ if and only if $p$ and $q$ are $\alpha$-compatible.
\end{proposition}
\begin{proof}
	Let $p = (\lambda, \vect{u}_1, \ldots, \vect{u}_d) \in \Var{P}_\alpha^\vect{k}$ and $q = (\mu, \vect{v}_1, \ldots, \vect{v}_d) \in \Var{P}_\alpha^\vect{k}$.

	The case where $\vect{u}_i = \vect{v}_i$ for all $i = 1, \dots, d$, so $M(p,q)=0$, corresponds to the case of straight line geodesics in \cref{prop_geodesics}.
	All such $p$ and $q$ can be connected by a straight line and since $M(p,q)=0$ they are always $\alpha$-compatible.
	Hence, the claim holds in this special case.

	The remainder of the proof treats the case $M(p,q) > 0$. For brevity, we let $M=M(p,q)$.
	\new{Recall from \cite[Lemma 4.4]{Chen1999} that the projection of a geodesic in a warped product onto the second factor decreases in length if the projected curve in the second factor is replaced by a minimizing geodesic. This implies that minimizing geodesics in a warped product project to minimizing geodesics.}

	\new{The geodesics passing through $p$ are described by \cref{lem_basic_geodesic_form}. In particular, a \emph{minimizing} geodesic $\sigma(s)=(\mu(s),\vect{v}_1(s),\dots,\vect{v}_d(s))$ that passes through both $p$ and $q$ must reach $q$ at $s = M$ because $\sigma$ is parameterized to unit spherical speed and the distance in $\Sp^{\vect{n},\vect{k}}$ between $(\vect{u}_1,\ldots,\vect{u}_d)$ and $(\vect{v}_1,\ldots,\vect{v}_d)$ is $M=M(p,q)$ from \cref{def_compatibility}.}
	\new{Plugging in this boundary condition $\sigma(M) = q$ allows us to determine the constant $B$ from \cref{lem_basic_geodesic_form}:}
	\begin{align*}
	\frac{\mu(0)}{\mu(M)}
 		= \frac{\lambda}{\mu}
 		= \frac{\cos(B + \alpha M)}{\cos(B)}
 		= \cos(\alpha M) - \sin(\alpha M) \tan(B),
	\end{align*}
	so that
	\begin{align}\label{eqn_formula_B}
 		B = \arctan\mleft( \cot(\alpha M) - \frac{\lambda \mu^{-1}}{\sin(\alpha M)} \mright),
	\end{align}
	provided that $\alpha M \ne \pi k$ for $k\in\mathbb{Z}$. \new{Recall from the statement of \cref{lem_basic_geodesic_form} that if no such real $B$ exists, then there is also no geodesic $\sigma$ with unit spherical speed and $\sigma(0)=p$ and $\sigma(q)=M$. We can exploit this to prove the equivalence between compatibility and the existence of minimizing geodesics.}
 	We prove the two implications next, \new{which will conclude the proof}.

	\paragraph{Geodesics $\Longrightarrow$ compatibility}
	Since $\sigma(s)$ exists by assumption, the solution $\mu(s) = \frac{\lambda \cos(B)}{\cos(\alpha s + B)} > 0$ for all $0 \le s \le M$.
	The inverse of the cosine has singularities at $B + \alpha s = \frac{\pi}{2} + k\pi$ for $k\in\mathbb{Z}$, so $\mu(s)$ cannot pass through them, and we must have 
	\(
 	-\frac{\pi}{2} + k\pi < B + \alpha s < \frac{\pi}{2} + k\pi
	\)
	for all $0 \le s \le M$. In particular, \new{plugging in the endpoint $s=0$ and combining with the bounds $-\frac{\pi}{2} < B < \frac{\pi}{2}$ from \cref{eqn_formula_B} yields $k=0$. So, $-\frac{\pi}{2} - B < \alpha s < \frac{\pi}{2} - B$ for all $0 \le s \le M$. Plugging in the other endpoint $s=M$ and exploiting $-\frac{\pi}{2} < B$, $M > 0$, and $\alpha > 0$ yields} $0 < \alpha M < \pi$. Hence, $p$ and $q$ are strictly $\alpha$-compatible.
	
	\paragraph{Compatibility $\Longrightarrow$ geodesics}
	As $p$ and $q$ are $\alpha$-compatible, $0 < \alpha M < \pi$, \new{so $B$ is given as in \cref{eqn_formula_B}.}
	Since $-\frac{\pi}{2} < B < \frac{\pi}{2}$, it follows from $0 < \alpha M < \pi$ that
	\begin{align*}
 	B + \alpha M = \arctan\mleft( \cot(\alpha M) - \frac{\lambda\mu^{-1}}{\sin(\alpha M)} \mright) + \alpha M
 	< \arctan\mleft( \cot(\alpha M) \mright) + \alpha M,
	\end{align*}
	because the arctangent is a monotonically increasing function and $\frac{\lambda\mu^{-1}}{\sin(\alpha M)} > 0$ when $\alpha M \in (0, \pi)$.
	Exploiting that $\cot(\alpha M) = \tan(\frac{\pi}{2} - \alpha M)$, we get $-\frac{\pi}{2} < B + \alpha M < \frac{\pi}{2}$ for all $0 < \alpha M < \pi$.
	Consequently, \new{$-\frac{\pi}{2} < B + \alpha s < \frac{\pi}{2}$ for all $0 \le s \le M$. Therefore,} $\mu(s)$ is smooth curve that is strictly positive for all $0 \le s \le M$. \new{The curves $\vect{v}_i(s)$ are great circles, and we can always choose a parameterization so that $\vect{v}_i(0)=\vect{u}_i$ and $\vect{v}_i(M)=\vect{v}_i$.}
	Hence, $\sigma \subset \Var{P}_\alpha^\vect{k}$ is a minimizing geodesic connecting $p$ and $q$.
\end{proof}

The foregoing result implies that by changing the $\alpha$-warped metric of $\RR_{>0} \times \Sp^{\vect{n},\vect{k}}$ we can modify whether two points $p$ and $q$ are connected by a minimizing geodesic.
Specifically, we have the following result.

\begin{proposition}\label{prop:pre-segre_geodesically_connected}
	The pre-Segre--Veronese manifold $\Var{P}_\alpha^\vect{k}$ is geodesically connected if and only if $0 < \alpha < \frac{1}{\sqrt{k_1+\cdots+k_d}}$.
\end{proposition}

\begin{proof}
	Let $p,q \in \Var{P}_\alpha^\vect{k}$. Then, $M^2(p,q) \le \sum_{i=1}^d k_i \pi^2$, with equality attained at $p=(\lambda,\vect{u}_1,\ldots,\vect{u}_d)$ and $q=(\lambda, -\vect{u}_1,\ldots,-\vect{u}_d)$.
	The result follows from \cref{prop_compatibility}.
\end{proof}

\subsection{Pre-Segre--Veronese manifolds} \label{sec_sub_preSVlog}

Having determined when two points in the pre-Segre--Veronese manifold $\Var{P}_\alpha^\vect{k}$ can be connected by a minimizing geodesic, we proceed by computing such a geodesic between two $\alpha$-compatible points.

\begin{theorem}[Logarithmic map of $\Var{P}_\alpha^\vect{k}$] \label{thm_logarithm_preSV}
	Let 
	\[
	p = (\lambda, \vect{u}_1, \ldots, \vect{u}_d) \in \Var{P}_\alpha^\vect{k} \quad\text{and}\quad 
	q = (\mu, \vect{v}_1, \ldots, \vect{v}_d) \in \Var{P}_\alpha^\vect{k}
	\] 
	be such that $\vect{u}_i \ne -\vect{v}_i$ for all $i=1,\dots,d$.
	Let $M = M(p, q)$ from \cref{def_compatibility}.
	If $p$ and $q$ are $\alpha$-compatible, then
	\begin{align*}
		\log_p(q) = \bigl( \mu \cos(\alpha M) - \lambda, \dot{\vect{u}}_1, \dots, \dot{\vect{u}}_d \bigr), 
	\end{align*}
	where \new{for $i=1,\dots,d$ we have}
	\begin{align*}
		\dot{\vect{u}}_i
		= \new{\frac{\mu \sin(\alpha M)}{\lambda \alpha M} \log^{\mathbb{S}}_{\vect{u}_i}(\vect{v}_i)}.
	\end{align*}
	In particular, if $M = 0$, then
	\(
		\log_{p}(q) = ( \mu - \lambda, 0, \dots, 0).
	\)
\end{theorem}
\begin{proof}
	The case $M=0$ follows immediately from \cref{prop_geodesics}. Hence, in the remainder of the proof we consider the case $M > 0$.

	\new{Based on the form of the exponential map in \cref{prop_geodesics},} we conjecture that
	\begin{align*}
		\log_{p}(q) =
		(\dot{\lambda}, \dot{\vect{u}}_1, \dots, \dot{\vect{u}}_d) = (C, D \log^\mathbb{S}_{\vect{u}_1}(\vect{v}_1), \dots, D \log^\mathbb{S}_{\vect{u}_d}(\vect{v}_d))
	\end{align*}
	for some $C \in \mathbb{R}$ and $D > 0$.
	Now we just have to solve $q ={} \exp_p(\log_p q)$ for $C$ and $D$.

	Recall the expression for $\exp_p$ from \cref{prop_geodesics}.
	The first component yields
	\begin{align}\label{eq:from_log_ansatz_1}
		(\lambda + C)^{2} + (\lambda \alpha D M)^{2} = \mu^{2},
	\end{align}
	having \new{used \cref{def_compatibility} and} $\|\dot{\vect{u}}_i\| = D \|\log^\mathbb{S}_{\vect{u}_i} (\vect{v}_i)\| = D \sphericalangle(\vect{u}_i,\vect{v}_i)$. \new{Note that $N$ from \cref{prop_geodesics} equals $D M$.}
	Since $M>0$, there exists at least one $1\le i\le d$ so that $\vect{u}_i \ne \vect{v}_i$.
	In all of these factors, we get \new{from \cref{prop_geodesics} that
	\[
	\vect{v}_i = \exp^\mathbb{S}_{\vect{u}_i}( a \dot{\vect{u}}_i ) = \exp^\mathbb{S}_{\vect{u}_i}( aD \log^\mathbb{S}_{\vect{u}_i}(\vect{v}_i) ),
	\]
	so that, applying $\log_{\vect{u}_i}$ on both sides, we obtain $aD = 1$.}
	Then, from the expression of $a$ in \cref{prop_geodesics}, we find
	\begin{align*}
		\frac{1}{D} = a = \frac{1}{\alpha D M} \left( \frac{\pi}{2} - \arctan \left( \frac{\lambda + C}{\alpha \lambda D M} \right)  \right).
	\end{align*}
	\new{Clearing $D^{-1}$, multiplying by $\alpha M$ and taking cotangents,} we find
	\begin{align}
		 \mathrm{cotan}(\alpha M) ={}& \frac{\lambda + C}{\alpha \lambda D M}.
		\label{eq:from_log_ansatz_2}
	\end{align}

	The square system of polynomial equations \cref{eq:from_log_ansatz_1,eq:from_log_ansatz_2} in $C$ and $D$ together with the constraint $D>0$ defines the intersection of a circle with a line in the half-plane $D>0$. Ignoring the $D>0$ constraint, the two intersection points are
	\begin{align*}
		C ={} \varsigma\mu \cos(\alpha M) - \lambda \quad\text{and}\quad
		D ={} \varsigma\frac{\mu \sin(\alpha M)}{\lambda \alpha M}, \quad\text{for } \varsigma \in \{-1,1\}.
	\end{align*}
	By $\alpha$-compatibility we have $0 < \alpha M < \pi$, and since we posited $D >0$, $\varsigma$ has to be $1$.

	With the above values for $C$ and $D$, one verifies \new{by direct computation} that $\exp_p \circ \log_p$ is \new{the identity, concluding the proof.}
\end{proof}

\begin{remark}
We caution that if $\vect{u}_i = -\vect{v}_i$ for some $i$, then the logarithmic map is not well defined.
However, $p$ and $q$ can still be connected by a geodesic.
It is just that this geodesic is not unique.
Choosing $\dot{\vect{u}}_i$ to be any vector orthogonal to $\vect{u}_i$ and with norm $\pi$ will produce such a geodesic.
\end{remark}

An immediate consequence of the previous theorem is the following expression for the distance between points in the pre-Segre--Veronese manifold.

\begin{corollary}[Distance between compatible points in $\Var{P}_\alpha^\vect{k}$] \label{prop_distance}
	Consider two points
	$p = (\lambda, \vect{u}_1, \ldots, \vect{u}_d) \in \Var{P}_\alpha^\vect{k}$ and $q = (\mu, \vect{v}_1, \ldots, \vect{v}_d) \in \Var{P}_\alpha^\vect{k}$ that are $\alpha$-compatible.
	Let $M := M(p,q)$ from \cref{def_compatibility}.
	Then,
	\begin{align}\label{eq:distance_formula}
 		\dist_{\Var{P}_\alpha^\vect{k}}(p,q)
 		= \new{\| \log_p(q) \|_{p}}
 		= \sqrt{\lambda^2 - 2 \lambda \mu \cos(\alpha M) + \mu^2}.
	\end{align}
\end{corollary}
\begin{proof}
\new{By exploiting \cref{eqn_warped_metric_preSV,thm_logarithm_preSV,def_compatibility}, this becomes a straightforward computation:
\begin{align*}
\norm{ \log_p(q) }_{p}^2
&= (\mu \cos(\alpha M) - \lambda )^2 + \alpha^2 \lambda^2 \sum_{i=1}^d k_i \frac{\mu^2 \sin^2(\alpha M)}{\lambda^2 \alpha^2 M^2} \norm{\log_{\vect{u}_i}(\vect{v}_i)}^2\\
&= \mu^2 \cos^2(\alpha M) - 2 \lambda \mu \cos(\alpha M) + \lambda^2 + \frac{\mu^2 \sin^2(\alpha M)}{M^2} \sum_{i=1}^d k_i \sphericalangle^2(\vect{u}_i,\vect{v}_i)\\
&= \mu^2 (\cos^2(\alpha M) + \sin^2(\alpha M)) - 2 \lambda \mu \cos(\alpha M) + \lambda^2,
\end{align*}
which concludes the proof.}
\end{proof}

A formula, equivalent to \cref{eq:distance_formula}, is
\begin{align}
 	\dist_{\Var{P}_\alpha^\vect{k}}(p, q) =
	\sqrt{(\lambda - \mu)^{2} + 4 \lambda \mu \sin(\alpha M / 2)},
\end{align}
which is more numerically stable for points that are close together.

The distance between $p$ and $q$ in $\Var{P}_\alpha^\vect{k}$ is thus given by the \emph{law of cosines} applied to an imaginary triangle where one side has length $\lambda$, the other side has length $\mu$, and the angle is the spherical distance $M(p,q)$ multiplied by the warping factor $\alpha$. We thus immediately obtain the following result.

\begin{corollary}\label{cor_monotonic_distance}
The distance $\dist_{\Var{P}_\alpha^\vect{k}}(p, q)$ is monotonically increasing as a function of the spherical distance $M(p,q)$ for $0 < \alpha M(p,q) < \pi$.
\end{corollary}

Interestingly, we can also determine the distance between \emph{incompatible} points, even though \emph{minimizing} geodesics do not exist between them.
While a minimizing geodesic does not exist, there exists a limiting piecewise smooth curve that is the infimizer of the length functional in \cref{eqn_length_functional}.
In \cref{fig_alphawarping_right}, for example, this limiting piecewise smooth curve is the straight line from $(1,0)$ to $(0,0)$, composed with the straight line from $(0,0)$ to $(0,1)$.

\begin{proposition}[Distance between incompatible points in $\Var{P}_\alpha^\vect{k}$]\label{prop_limiting_curve}
	Two $\alpha$-in\-com\-pa\-ti\-ble points $p = (\lambda, \vect{u}_1, \ldots, \vect{u}_d) \in \Var{P}_\alpha^\vect{k}$ and $q = (\mu, \vect{v}_1, \ldots, \vect{v}_d) \in \Var{P}_\alpha^\vect{k}$ are at
	\begin{align*}
		\dist_{\Var{P}_\alpha^\vect{k}}(p,q) = \lambda + \mu.
	\end{align*}
\end{proposition}
\begin{proof}
	\new{We prove the distance is lower and upper bounded by $\lambda +\mu$.}

	\paragraph{Upper bound} \new{The distance between $p$ and $q$,
	\(
	\dist_{\Var{P}_\alpha^\vect{k}}(p,q) := \inf_{\gamma} \ell(\gamma),
	\)
	where $\gamma$ ranges over all piecewise smooth curves in $\Var{P}_\alpha^\vect{k}$ with $\gamma(0)=p$ and $\gamma(1)=q$, can be upper bounded as follows. Let $\epsilon > 0$ be arbitrary and}
	consider the curve that goes first in a straight line from $p = (\lambda, \vect{u}_1, \dots, \vect{u}_d)$ to $(\epsilon, \vect{u}_1, \dots, \vect{u}_d)$, then continues on a circular arc to $(\epsilon, \vect{v}_1, \dots, \vect{v}_d)$, and finally continues in a straight line to $q = (\mu, \vect{v}_1, \dots, \vect{v}_d)$.
	It has length $\lambda + \mu + (\alpha M - 2) \epsilon$, \new{for all $\epsilon > 0$. Hence, $\dist_{\Var{P}_\alpha^\vect{k}}(p,q) \le \lambda + \mu$, establishing the upper bound.}%

	\paragraph{Lower bound} \new{The set
	\[
	 \Var{P}' := \{ q' \in \Var{P}_{\alpha}^\vect{k} \mid \dist_{\Var{P}_\alpha^\vect{k}}(p,q') < \lambda + \mu + 1 \}
	\]
	is an open subset of $\Var{P}_\alpha^\vect{k}$.} Let $\gamma_n$ be a sequence of unit speed curves in $\Var{P}_\alpha^\vect{k}$ that connect two incompatible points $p$ and $q$, and whose lengths approach the infimum. \new{For sufficiently large $n$, we have $\gamma_n \subset \Var{P}'$. Otherwise, $\gamma_n$ contains a point of $\Var{P}_\alpha^\vect{k} \setminus \Var{P}'$, which lies at distance at least $\lambda + \mu + 1$ from $p$, so $\ell(\gamma_n) \ge \lambda + \mu + 1 \ge \dist_{\Var{P}_\alpha^\vect{k}}(p,q) +1$.}

	\new{The limit $\gamma$ of the sequence of $\gamma_n$'s cannot be contained in $\Var{P}'$, however, as this would show there is a piecewise smooth curve in $\Var{P}_\alpha^\vect{k}$ realizing the distance between $p$ and $q$, i.e., $p$ and $q$ would be $\alpha$-compatible by \cref{prop_compatibility}, contradicting their assumed incompatibility. As $\gamma_n \subset \Var{P}'$, the limit $\gamma$ must lie in the (metric) closure
	\[
	 \Var{C} = \Var{P'} \cup \{0\} \cup \{ q' \in \Var{P}_{\alpha}^\vect{k} \mid \dist_{\Var{P}_\alpha^\vect{k}}(p,q') = \lambda+\mu+1 \}
	\]
	of $\Var{P}'$; $0$ is in the closure since its distance to $p$ is at most $\lambda$. Since $\ell(\gamma) = \dist_{\Var{P}_\alpha^\vect{k}}(p,q)$, $\gamma \subset \Var{P}' \cup \{0\}$, so $\gamma$ must contain $0$. As $\gamma_n \to \gamma$, this entails that for a fixed $n$, there is a point} $a = (\epsilon_n, \vect{a}_1, \dots, \vect{a}_d) \in \Var{P}'$ in $\gamma_n$ \new{with $\epsilon_n$ minimal.} Let $\gamma^{(1)}$ denote the part of $\gamma_n$ connecting $p$ and $a$ and let $\gamma^{(2)}$ be the part of $\gamma_n$ connecting $a$ and $q$.
	Then,
	\begin{align*}
		\ell(\gamma_n) = \ell(\gamma^{(1)}) + \ell(\gamma^{(2)}).
	\end{align*}
	Moreover, if $\gamma^{(1)}(s) = (\nu(s), \vect{x}_1(s), \dots, \vect{x}_d(s))$, then
	\begin{align*}
		\ell(\gamma^{(1)}) = \int \sqrt{\nu'(s)^2 + \alpha^2 \nu(s)^2} \dd{s}
			\geq \int \sqrt{\nu'(s)^2} \dd{s}
			= \lambda - \epsilon_n.
	\end{align*}
	A similar argument shows that $\ell(\gamma^{(2)}) \geq \mu - \epsilon_n$.
	Specifically, for every $n$, there exists an $\epsilon_n$ such that $\ell(\gamma_n) \geq \lambda + \mu - 2 \epsilon_n$.
	\new{Letting $n\to\infty$, we have $\epsilon_n \to 0$ because $\gamma$ contains $0$, which shows that $\dist_{\Var{P}_\alpha^\vect{k}}(p,q) \ge \lambda + \mu$.}
\end{proof}

\subsection{Segre--Veronese manifolds} \label{sec_sub_SVlog}

By leveraging \cref{lem_basic_covering_property,thm_logarithm_preSV}, minimizing geodesics in the Segre--Veronese manifold $\Var{S}_\alpha^\vect{k}$ that connect two rank-$1$ tensors are readily obtained.
Indeed, a minimizing geodesic connecting $\tensor{P}$, $\tensor{Q} \in \Var{S}_\alpha^\vect{k}$ is the pushforward of any geodesic whose length coincides with
\begin{align*}
	\min_{\substack{p\in\otimes^{-1}(\tensor{P}),\, q\in\otimes^{-1}(\tensor{Q})}} \dist_{\Var{P}_\alpha^\vect{k}}(p, q).
\end{align*}
\Cref{cor_monotonic_distance} then leads us naturally to the following terminology.

\begin{definition}[Matched representatives] \label{def_matched_representatives}
	Let $\tensor{P}$, $\tensor{Q} \in \Var{S}_\alpha^\vect{k}$.
	We say that the representatives $p^\star \in \otimes^{-1}(\tensor{P})$ and $q^\star \in \otimes^{-1}(\tensor{Q})$ are \emph{matched} if
	\begin{align*}
		(p^\star, q^\star) \in \underset{\substack{p \in \otimes^{-1}(\tensor{P}),\, q \in \otimes^{-1}(\tensor{Q})}}{\arg\min} M(p,q),
	\end{align*}
	where $M$ is the distance from \cref{def_compatibility}.
\end{definition}

We are now ready to state and prove the main result about minimizing geodesics connecting two rank-$1$ tensors in the $\alpha$-warped Segre--Veronese manifold.

\begin{theorem}[The logarithmic map of $\Var{S}_\alpha^\vect{k}$]\label{thm_logarithm_SV}
	Let $\tensor{P}$, $\tensor{Q} \in \Var{S}_\alpha^\vect{k}$ be rank-$1$ tensors.
	If $p = (\lambda, \vect{u}_1, \ldots, \vect{u}_d) \in \otimes^{-1}(\tensor{P})$ and $q = (\mu, \vect{v}_1, \ldots, \vect{v}_d) \in \otimes^{-1}(\tensor{Q})$ are matched, $\alpha$-compatible representatives, then
	\begin{align*}
		\log_\tensor{P} \tensor{Q} = (\deriv_{p} \otimes)(\log_p q),
	\end{align*}
	\new{where $\deriv_p\otimes: \Tang_{p} {\Var{P}_\alpha^\vect{k}} \to \Tang_{\tensor{P}} {\Var{S}_\alpha^\vect{k}}$ is the differential of $\otimes$.}
\end{theorem}
\begin{proof}
	Since $p$ and $q$ are compatible, there is by \cref{prop_compatibility} a minimizing geodesic between them.
	To pull this back to a minimizing geodesic in $\Var{S}_\alpha^\vect{k}$, we need to know that inequalities between the lengths are not reversed.
	That is, if $\gamma$ and $\tilde{\gamma}$ are geodesics in $\Var{P}_\alpha^\vect{k}$ such that $(\otimes \circ \gamma)(0) = (\otimes \circ \tilde{\gamma})(0) = \tensor{P}$, $(\otimes \circ \gamma)(1) = (\otimes\circ \tilde{\gamma})(1) = \tensor{Q}$, and $\ell(\gamma) \leq \ell(\tilde{\gamma})$, then we need to show that $\ell(\otimes \circ \gamma) \leq \ell(\otimes \circ \tilde{\gamma})$.
	This implication indeed holds as a direct consequence of \cref{cor_monotonic_distance,lem_riemannian_covering,lem_basic_covering_property}.
\end{proof}

The following two results follow similarly.

\begin{proposition}\label{prop:existence_of_minimizing_geodesics}
	Two rank-1 tensors $\tensor{P}$ and $\tensor{Q} \in \Var{S}_{\alpha}^{\vect{k}}$ can be connected by a minimizing geodesic in $\Var{S}_{\alpha}^{\vect{k}}$ if and only if $\tensor{P}$ and $\tensor{Q}$ have matched $\alpha$-compatible \new{representatives}.
\end{proposition}

\begin{proposition}[Distance between points in $\Var{S}_{\alpha}^{\vect{k}}$]\label{prop:distance_formula}
	Let $\tensor{P}$, $\tensor{Q} \in \Var{S}_{\alpha}^{\vect{k}}$.
	If $p = (\lambda, \vect{u}_{1}, \dots, \vect{u}_d) \in \otimes^{-1}(\tensor{P})$ and $q = (\mu, \vect{v}_{1}, \dots, \vect{v}_{d}) \in \otimes^{-1}(\tensor{Q})$ are matched, $\alpha$-compatible representatives, then
	\begin{align*}
		\operatorname{dist}_{\Var{S}_{\alpha}^{\vect{k}}}(\tensor{P}, \tensor{Q}) = \operatorname{dist}_{\Var{P}_{\alpha}^{\vect{k}}}(p, q).
	\end{align*}
\end{proposition}

The straightforward approach to determine matched representatives of $\tensor{P}$, $\tensor{Q}\in\Var{S}_\alpha^\vect{k}$ consists of choosing an arbitrary $p^\star\in\otimes^{-1}(\tensor{P})$ and then computing all distances $M(p^\star,q)$ for all $q\in\otimes^{-1}(\tensor{Q})$ and selecting a minimum.
By \cref{lem_riemannian_covering}, the possible $q$ can be labeled by $(\sigma_1, \dots, \sigma_d) \in \{\, -1, 1 \,\}^{d}$ such that $\sigma_1^{k_1} \cdots \sigma_d^{k_d} = 1$.
Thus, this strategy has a computational complexity of $\mathcal{O}(2^{d} \dim \Var{S}_\alpha^\vect{k})$, where $\mathcal{O}(\dim \Var{S}_\alpha^\vect{k})$ is the asymptotic cost for computing one distance $M(p^\star,q)$.
We now explain how a matched representative can be constructed efficiently in $\mathcal{O}(d \dim\Var{S}_\alpha^\vect{k})$ operations.

Let $q_{\mathrm{ref}} \in \otimes^{-1}(\Var{Q})$ be a reference preimage.
We then observe that
\begin{align*}
	M(p^{\star}, q)^{2} = M(p^{\star}, q_{\mathrm{ref}})^{2} + \sum_{i = 1}^{d} \frac{1 - \sigma_i}{2} \Delta_i
\end{align*}
with $\Delta_i = k_i ((\sphericalangle_i - \pi)^{2} -\sphericalangle_i^{2})$, where $\sphericalangle_i$ is the angle between the $i$th components of $p^{\star}$ and $q_{\mathrm{ref}}$.
Minimizing $M(p^{\star}, q)$ thus corresponds to maximizing $\sum_{i = 1}^{d} {\sigma_i} \Delta_i$ over $\sigma_i \in \{-1, 1\}$.
An unconstrained optimum is
\begin{align}\label{eq:minimizing_sigmas}
	\sigma_i =
	\begin{cases}
		1, &\text{ if $\Delta_i$ is nonnegative},\\
		-1, &\text{otherwise}.
	\end{cases}
\end{align}
If this assignment also satisfies the constraint $\sigma_1^{k_1} \cdots \sigma_d^{k_d} = 1$, so that $q \in \otimes^{-1}(\tensor{Q})$, then we are done.
Otherwise, there is at least one smallest $\abs{\Delta_\ell}$ among the odd $k_\ell$'s.
The sign of the corresponding $\sigma_\ell$ should be flipped.

\begin{proposition}[Matchmaking]\label{prop_matching_algorithm}
	The algorithm described in the previous paragraph produces matched representatives.
\end{proposition}
\begin{proof}
A detailed proof is given in \cref{sub:proof_of_matching_algorithm}.
\end{proof}

The construction of the matchmaking algorithm suggests that determining an analogue of \cref{prop:pre-segre_geodesically_connected} is more complicated for the Segre--Veronese manifold.
Essentially, we have to consider the following problem: Given $p^{\star} \in \Var{P}_{\alpha}^{\vect{k}}$, what is the set of $q \in \Var{P}_{\alpha}^{\vect{k}}$ such that the corresponding matched representatives $q^{\star} \in \Var{P}_{\alpha}^{\vect{k}}$ are $\alpha$-compatible with $p^{\star}$, and when is this set the whole of $\Var{P}_{\alpha}^{\vect{k}}$? By our calculations, this question can be reduced to a combinatorial optimization problem. As we were unable to produce a closed form of its solution \new{in all cases}, we first present the following general result.

\begin{proposition}\label{prop_geo_connected}
	The Segre--Veronese manifold $\Var{S}_\alpha^\vect{k}$ is geodesically connected if $0 < \alpha < 1 / \sqrt{k_1 + \dots + k_d}$ and is not geodesically connected if $2 / \sqrt{k_1 + \dots + k_d} \leq \alpha$.
\end{proposition}
\begin{proof}
	The first statement is a corollary of \cref{lem_riemannian_covering,prop:pre-segre_geodesically_connected}.

	To prove the second statement, assume $\alpha \geq 2 / \sqrt{k_1 + \dots + k_d}$ and pick $p$, $q \in \Var{P}_\alpha^\vect{k}$ with angles $\sphericalangle_1 = \dots = \sphericalangle_d = \frac{\pi}{2}$ between their components.
	These two representatives are matched.
	However, using
	\(
		M(p, q)^{2} = (k_1 + \dots + k_d) \frac{\pi^{2}}{4}
	\)
	in \cref{def_compatibility} shows that they are not $\alpha$-compatible.
\end{proof}

For $1 / \sqrt{k_1 + \dots + k_d} < \alpha < 2 / \sqrt{k_1 + \dots + k_d}$, geodesic connectedness is more complicated and depends also on the parity of the $k_i$'s. \new{For example, we have the following basic special case.}

\new{\begin{proposition}\label{prop_geo_connected2}
	Let $\vect{k} = (k_1, \dots, k_d)$ with $k_i$ even for all $i$.
	Then, the Segre--Veronese manifold $\mathcal{S}_{\alpha}^{\vect{k}}$ is geodesically connected if and only if $0 < \alpha < \frac{2}{\sqrt{k_1 + \dots + k_d}}$.
\end{proposition}
\begin{proof}
	If the $k_i$'s are even, then $\sigma_1^{k_1} \cdots \sigma_d^{k_d} = 1$ is automatically satisfied.
	The condition that two representatives $p$ and $q$ are matched is then equivalent to the requirement that the component angles satisfy $0 \le \sphericalangle_i \leq \frac{\pi}{2}$. The largest value that $M(p, q)^{2}$ can attain over all $p,q\in\Var{P}_\alpha^\vect{k}$ is then
	$(k_1 + \dots + k_d) \frac{\pi^{2}}{4}$, by choosing all angles equal to $\pi/2$ as in the proof of \cref{prop_geo_connected}. Hence, all points of $\Var{P}_\alpha^\vect{k}$ are compatible if and only if $0 < \alpha < 2 / \sqrt{k_1+\dots+k_d}$, concluding the proof.
\end{proof}}

\new{The previous result shows that the lower bound on $\alpha$ in \cref{prop_geo_connected} can be sharp, while the upper bound on $\alpha$ in \cref{prop_geo_connected} can fail to be sharp.}
A full analysis is beyond the scope of this article.

\section{Sectional curvature}%
\label{sec_sectional_curvatures}

Formulae for the sectional curvature of a manifold are useful in many situations, for example when approximating or compressing data that live in said manifold~\cite{JVVV2024,Diepeveen23,Zimmermann20}, or to ensure uniqueness of the Fr\'echet mean~\cite{Karcher1977,Kendall1990,Afsari2011}, as required in the experiment from \cref{sec_applications}. We compute these curvatures here for the $\alpha$-warped Segre--Veronese manifolds.

Consider first the $\alpha$-warped Segre manifold $\Var{S}_{\alpha} = \Var{S}_{\alpha}^{(1, \dots, 1)}$ \new{with corresponding pre-Segre manifold $\Var{P}_\alpha$.}
To compute its curvature, we recall the following general result about warped products.

\begin{proposition}[O'Neill {\cite[7.42]{ONeill1983}}]\label{prop:curvature_of_warped_product}
	For $(a, b) \in \Var{M} \times_f \Var{N}$, let $\dot{\vect x}$, $\dot{\vect y}$, $\dot{\vect z} \in \Tang_a \Var{M}$ and $\dot{\vect u}$, $\dot{\vect v}$, $\dot{\vect w} \in \Tang_b \Var{N}$ be tangent vectors, which naturally lift to $\Tang_{(a, b)}(\Var{M} \times_f \Var{N})$.
	Let $R_\Var{M}$ and $R_\Var{N}$  be the Riemann \new{curvature} tensors on $\Var{M}$ and $\Var{N}$, respectively.
Then, the Riemann \new{curvature} tensor $R$ on $\Var{M} \times_f \Var{N}$ is described by:\footnote{%
	Note that we use a definition of $R$ that is consistent with~\cite{Lee1997} rather than~\cite{ONeill1983}.
	The result is that some signs differ in our expressions from theirs.}
	\begin{enumerate}[itemsep=3pt,parsep=4pt]
		\item $\displaystyle R(\dot{\vect x}, \dot{\vect y}) \dot{\vect z} = R_{\Var{M}}(\dot{\vect x}, \dot{\vect y}) \dot{\vect z}$,
		\item $\displaystyle R(\dot{\vect u}, \dot{\vect x}) \dot{\vect y} = -\frac{1}{f} \operatorname{Hess}[f](\dot{\vect x}, \dot{\vect y}) \dot{\vect u}$,
		\item $\displaystyle R(\dot{\vect x}, \dot{\vect y}) \dot{\vect u} = 0$,
		\item $\displaystyle R(\dot{\vect u}, \dot{\vect v}) \dot{\vect x} = 0$,
		\item $R(\dot{\vect x}, \dot{\vect u}) \dot{\vect v} = -\frac{\innerproduct{\dot{\vect u}}{\dot{\vect v}}}{f} \nabla_{\dot{\vect x}} (\operatorname{grad} f)$,
		\item $R(\dot{\vect u}, \dot{\vect v}) \dot{\vect w} = R_{\Var{N}}(\dot{\vect u}, \dot{\vect v}) \dot{\vect w} + \frac{\norm{\operatorname{grad} f}^{2}}{f^{2}} \left( \innerproduct{\dot{\vect u}}{\dot{\vect w}} \dot{\vect v} - \innerproduct{\dot{\vect v}}{\dot{\vect w}} \dot{\vect u} \right)$.
	\end{enumerate}
	Herein, $\innerproduct{\cdot}{\cdot}$ is the inner product on $\Tang_{(a,b)} (\Var{M} \times_{f} \Var{N})$, $\operatorname{grad} f$ is the Riemannian gradient of $f$, $\operatorname{Hess}[f]$ is the Riemannian Hessian of $f$, and $\nabla$ is the Levi--Civita connection.
\end{proposition}

For $\Var{S}_{\alpha}$, which locally looks like the warped product manifold $\Var{P}_{\alpha}$, these expressions simplify considerably. 
Applying \cref{prop:curvature_of_warped_product} to $\Var{N} = \Sp^{\vect{n},\vect{1}} = \mathbb{S}^{n_1 - 1} \times \dots \times \mathbb{S}^{n_d - 1}$ and $\Var{M} = \mathbb{R}_{> 0}$ while observing that $\operatorname{grad} f = \alpha$ and $R_{\Var{M}} = 0$ yields the next result.

\begin{corollary}
	In the notation of \cref{prop:curvature_of_warped_product}, the Riemann \new{curvature} tensor $R$ of $\Var{S}_{\alpha}$ at $\lambda \vect{u}_1\otimes\dots\otimes\vect{u}_d$ satisfies
	\begin{align*}
		R(\dot{\vect u}, \dot{\vect v}) \dot{\vect w} = R_{\Sp^{\vect{n},\vect{1}}}(\dot{\vect u}, \dot{\vect v}) \dot{\vect w} + \frac{1}{\lambda^{2}} (\innerproduct{\dot{\vect u}}{\dot{\vect w}} \dot{\vect v} - \innerproduct{\dot{\vect v}}{\dot{\vect w}} \dot{\vect u}),
	\end{align*}
	and it is $0$ in the remaining directions.
\end{corollary}

As sectional curvatures can be defined in terms of the Riemann \new{curvature} tensor, we quickly find the next consequence of the last result.

\begin{corollary}\label{cor:segre_sectional_curvature}
	For $(\vect{u}_1, \dots, \vect{u}_d) \in \Sp^{\vect{n},\vect{1}}$, let the lifts of the tangent vectors $\dot{\vect{u}}_i \in \Tang_{\vect{u}_i} \mathbb{S}^{n_i - 1}$ and $\dot{\vect v}_j \in \Tang_{\vect{u}_j} \mathbb{S}^{n_j - 1}$ be orthonormal with respect to $\langle \cdot, \cdot \rangle_{\Var{P}_{\alpha}}$.
	The sectional curvature $K$ of $\Var{S}_{\alpha}$ at $\lambda \vect{u}_1\otimes\dots\otimes\vect{u}_d$ satisfies
\begin{align*}
	K(\dot{\vect u}_i, \dot{\vect v}_j) =
	\begin{cases}
		\frac{1 - \alpha^{2}}{\alpha^{2} \lambda^{2}}, &\text{if $i = j$},\\
		-\frac{1}{\lambda^{2}}, &\text{otherwise},
	\end{cases}
\end{align*}
and $K = 0$ in all remaining directions.
\end{corollary}
\begin{proof}
	Recall that $\innerproduct{\dot{\vect u}_i}{\dot{\vect v}_j}_{\Var{P}_{\alpha}} = \alpha^{2} \lambda^2 \innerproduct{\dot{\vect u}_i}{\dot{\vect v}_j}_{\Sp^{\vect{n},\vect{1}}}$.
	We compute 
	\begin{align*}
		K(\dot{\vect u}_i, \dot{\vect v}_j) ={}& \innerproduct{R(\dot{\vect u}_i, \dot{\vect v}_j) \dot{\vect v}_j}{\dot{\vect u}_i}_{\Var{P}_{\alpha}} \nonumber\\
		={}& \innerproduct{R_{\Sp^{\vect{n},\vect{1}}}(\dot{\vect u}_i, \dot{\vect v}_j) \dot{\vect v}_j}{\dot{\vect u}_i}_{\Var{P}_{\alpha}} + \frac{1}{\lambda^{2}} \mleft( \innerproduct{\dot{\vect u}_i}{\dot{\vect v}_j}_{\Var{P}_{\alpha}}^{2} - \innerproduct{\dot{\vect v}_j}{\dot{\vect v}_j}_{\Var{P}_{\alpha}} \innerproduct{\dot{\vect u}_i}{\dot{\vect u}_i}_{\Var{P}_{\alpha}} \mright) \nonumber\\
		={}& \alpha^{2} \lambda^{2} \innerproduct{R_{\Sp^{\vect{n},\vect{1}}}(\dot{\vect u}_i, \dot{\vect v}_j) \dot{\vect v}_j}{\dot{\vect u}_i}_{\Sp^{\vect{n},\vect{1}}} - \frac{1}{\lambda^{2}} \nonumber\\
		={}& \frac{\alpha^{2} \lambda^{2}}{\alpha^{4} \lambda^{4}} \innerproduct{R_{\Sp^{\vect{n},\vect{1}}}(\alpha \lambda \dot{\vect u}_i, \alpha \lambda \dot{\vect v}_j) \alpha \lambda \dot{\vect v}_j}{\alpha \lambda \dot{\vect u}_i}_{\Sp^{\vect{n},\vect{1}}} - \frac{1}{\lambda^{2}} \nonumber\\
		={}& \frac{1}{\alpha^{2} \lambda^{2}} K_{\Sp^{\vect{n},\vect{1}}}(\alpha \lambda \dot{\vect u}_i, \alpha \lambda \dot{\vect v}_j) - \frac{1}{\lambda^{2}} \nonumber\\
		={}& \frac{1}{\alpha^{2} \lambda^{2}} \delta_{i j} - \frac{1}{\lambda^{2}}.
	\end{align*}	
	In all other directions the sectional curvature is $0$, because $R = 0$ in these directions.
\end{proof}

The Segre--Veronese manifold $\Var{S}_{\alpha}^{\vect{k}}$, $\vect{k} = (k_1, \dots, k_d)$, is a Riemannian submanifold of the Segre manifold $\Var{S}_{\alpha}^{\vect{1}}$, where $\vect{1}$ is a vector of $k_1 + \dots + k_d$ ones.
If we look closely at the expressions in \cref{prop_geodesics,cor_exponential_map} for geodesics in $\Var{S}_{\alpha}^{\vect{1}}$, we see that starting out in a tangent direction to $\Var{S}_\alpha^\vect{k} \subset \Var{S}_\alpha^\vect{1}$ produces a path wholly contained in $\Var{S}_\alpha^\vect{k}$.
We have thus proved the following result.

\begin{proposition}
	The Segre--Veronese manifold $\Var{S}_{\alpha}^{\vect{k}}$ is a totally geodesic submanifold of the Segre manifold $\Var{S}_{\alpha}^\vect{1}$, and so its sectional curvature is described by restricting $K$ in \cref{cor:segre_sectional_curvature} to the appropriate tangent subspace.
\end{proposition}

\section{Numerical experiment} \label{sec_applications}
\new{We conclude with a numerical illustration of an example application that benefits from knowledge of the exponential and logarithmic maps of the $\alpha$-warped Segre manifold, derived in \cref{cor_exponential_map,thm_logarithm_SV}. This application will also illustrate their computational efficacy.}

The $\alpha$-warped \new{geometries of} Segre--Veronese manifolds $\Var{S}_\alpha^{\vect{k}}$ were implemented in Julia, including \new{the exponential map, logarithmic map}, and the geodesic distance. A modified version of our implementation of these geometries is available as part of the Manifolds.jl library~\cite{Axen24}.
\new{The Julia source code for the experiments below is available publicly at \url{https://gitlab.kuleuven.be/numa/public/alpha-warped-experiments}.}

Recall the \emph{consensus aggregation problem} from Chafamo, Shanmugam, and Tokcan~\cite{CST2023}:
Given $M$ approximate \emph{tensor rank decompositions} of the same tensor $\tensor{A} \in \RR^{n_1 \times\dots\times n_d}$, namely
\begin{align*}
	\tensor{A} \approx \tensor{A}^{(m)} :=
	\sum_{i=1}^r \tensor{A}_i^{(m)} :=
	\sum_{i = 1}^{r} \vect{a}_{1i}^{(m)} \otimes \dots \otimes \vect{a}_{di}^{(m)},\quad m = 1,\, \dots,\, M,
\end{align*}
estimate the true decomposition of $\tensor{A}$.
A consensus aggregation algorithm was proposed in \cite{CST2023}; we give a simplified description of the essential ingredients of their method. 
While tensor rank decompositions usually have generically unique decompositions into a set of rank-$1$ tensors \cite{MM2024,TBC2024,Ballico2024}, to perform aggregation one needs to determine how these rank-$1$ tensors match up across multiple decompositions.
For this reason, the rank-$1$ tensors $\tensor{A}_i^{(m)}$ are first matched using a ($k$-means) clustering algorithm in~\cite{CST2023}.
This strategy can work well if the decomposition problem is \emph{well conditioned}, so that $\tensor{A}^{(m)}$ being close to $\tensor{A}$ will imply that the corresponding rank-$1$ tensors are close too~\cite{Breiding18}.
After matching up the rank-$1$ tensors, the results are aggregated in~\cite{CST2023} by setting
\begin{align*}
	\widehat{\vect{a}}_{i}^{l} ={}& \operatorname{median}( \vect{a}^{(m)}_{li} : m = 1,\, \dots,\, M ), \quad i=1,\ldots,r,\; l=1,\ldots,d.
\end{align*}
The median \new{of these vectors} is taken elementwise.
Note that it is possible for $\widehat{\vect{a}}_{i}^{l}$ to have different values if different representatives are chosen.
This is circumvented in \cite{CST2023} by aggregating positive data, for which there is a natural, positive, representative.
The aggregated tensor decomposition is then
\begin{align*}
	\tensor{A} \approx \widehat{\tensor{A}} := \sum_{i = 1}^{r} \widehat{\vect{a}}_i^{1} \otimes \dots \otimes \widehat{\vect{a}}_i^{d}.
\end{align*}

\new{The method from \cite{CST2023} essentially takes a Euclidean midpoint of multiple points in spheres of different radii with the goal of averaging rank-$1$ tensors in the Segre manifold. From a geometric viewpoint, it is more natural to exploit the intrinsic geometry of the Segre manifold in its Euclidean metric, i.e., $\Var{S}_1^{(1,\dots,1)}$, rather than using the product Euclidean metric on the pre-Segre manifold $\Var{P}_1^{(1,\dots,1)}$. Recall from \cref{lem_riemannian_covering} that pre-Segre and Segre manifolds are locally isometric under the \emph{warped} product metric on the pre-Segre manifold, but not under the \emph{standard} product metric that \cite{CST2023} use. Therefore, we propose here to perform consensus aggregation by}
estimating each rank-$1$ tensor with the \emph{Fr\'echet mean} on the warped Segre manifold $\Var{S}_{\alpha}=\Var{S}_\alpha^{(1,\dots,1)}$.
Recall that the Fr\'echet mean or \emph{Riemannian center of mass} of points \new{$x_1,\dots,x_n \in \Var{M}$} in a Riemannian manifold $\Var{M}$ is \new{defined as}
\begin{align*}
	\operatorname{mean}_{\Var{M}}(x_1,\dots,x_n) = \underset{p\in\Var{M}}{\arg\min} \sum_{i = 1}^{n} \dist_{\Var{M}}(p, \new{x_i})^{2}.
\end{align*}
Its uniqueness has been investigated among others in~\cite{Karcher1977,Kendall1990,Afsari2011}.
For manifolds with negative sectional curvatures, the Fr\'echet mean is unique and the above defines a valid function.
Hence, we propose to estimate $\tensor{A}$ instead as
\begin{align*}
\tensor{A} \approx \widetilde{\tensor{A}} := \sum_{i=1}^r \widetilde{\tensor{A}}_i, \quad\text{where }
\widetilde{\tensor{A}}_i = \operatorname{mean}_{\Var{S}_\alpha}( \tensor{A}_i^{(1)}, \dots, \tensor{A}_i^{(M)}),\, i=1,\dots,r.
\end{align*}
To ensure that all points are \new{compatible} by \cref{prop_geo_connected}, \new{one can} choose $\alpha < 1 / \sqrt{d}$. This guarantees that the Fr\'echet mean is well defined \new{for all inputs}. \new{Note that significant warping may not be necessary if all inputs to the Fr\'echet mean are sufficiently close to one another in the angular distance of \cref{def_compatibility}.}

We reproduce the \new{first} experiment from \cite{CST2023}, using the authors' code with the same parameter values and initialization seed. \new{This synthetic experiment involves generating a \emph{reference} rank-$9$ tensor rank decomposition of a $40 \times 20 \times 2000$ real-valued tensor $\tensor{A}$ with positive elements. Twenty noisy versions of $\tensor{A}$ are created by adding Poisson noise as described in \cite{CST2023}. They are then decomposed with the zero-inflated Poisson tensor factorization (ZIPTF) method from \cite{CST2023} from $10$ random starting points. We compare the following three methods of aggregating the results:}
\begin{enumerate}
 \item[S.] \new{No aggregation, i.e., the first ZIPTF decomposition of the first noisy tensor is used as ``consensus'';}
 \item[C.] C-ZIPTF is the consensus aggregation of \new{all corresponding rank-$1$ tensors in the $10$ decompositions of one noisy tensor} by the method of medians, \new{as described previously and in \cite{CST2023}}; and
 \item[$\alpha$.] F-ZIPTF \new{operates like C-ZIPTF but aggregates the rank-$1$ tensors} by our proposed method of Fr\'echet means. \new{The latter is computed in the $\alpha$-warped geometry of the Segre manifold for various choices of $\alpha$.}
\end{enumerate}
The Fr\'echet mean is approximated by successive geodesic interpolation, \new{as implemented in Manifolds.jl}; see for example~\cite{Chakraborty20,Chakraborty15,Chakraborty19,Cheng16}.
\new{To measure the accuracy of these methods, we measure the error between the reference tensor rank decomposition $\tensor{A} = \tensor{A}_1 + \dots + \tensor{A}_9$ and an approximating decomposition $\widehat{\tensor{A}} = \widehat{\tensor{A}}_1 +\dots+ \widehat{\tensor{A}}_9 \approx \tensor{A}$ as the mean of the relative Riemannian distances between the rank-$1$ tensors in both decompositions:
\begin{align}\label{eqn_mean_rel_dist}
 \mathrm{error}( (\tensor{A}_1,\dots,\tensor{A}_9), (\widehat{\tensor{A}}_1,\dots,\widehat{\tensor{A}}_9 ) := \frac{1}{9} \sum_{i=1}^9 \frac{\dist_{\Var{S}_1}(\tensor{A}_i, \widehat{\tensor{A}}_i)}{\|\tensor{A}_i\|_F},
\end{align}
where $\norm{\cdot}_F$ is the Frobenius norm.
The rank-$1$ tensors are first optimally matched before computing this error.}

\begin{figure}[tb]
	\centering
	\includegraphics[height=17em]{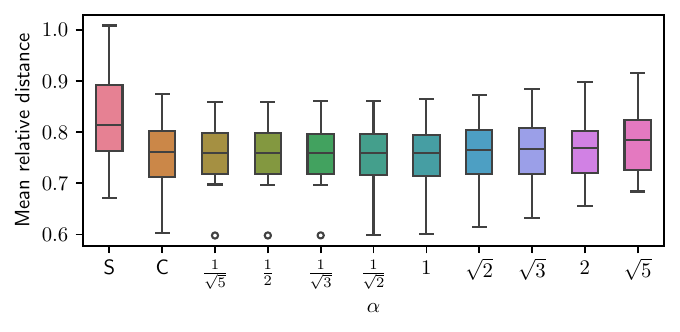}
	\caption{
		\new{The mean relative distance from \cref{eqn_mean_rel_dist} between the reference tensor rank decomposition and the approximate decompositions computed from noisy samples of this tensor rank decomposition using a single ZIPTF \cite{CST2023}, the consensus C-ZIPTF \cite{CST2023}, and the Fr\'echet F-ZIPTF in several $\alpha$-warped geometries of the Segre manifold (this article) on the decomposition problem setup described in \cref{sec_applications}. In the figure's horizontal axis, ``S'' refers to a single ZIPTF, ``C'' to C-ZIPTF, and the numerical values $\alpha=\frac{1}{\sqrt{5}}, \dots, \frac{1}{\sqrt{2}}, 1, \sqrt{2}, \dots, \sqrt{5},$ correspond to F-ZIPTF where the Fr\'echet mean was computed in the $\alpha$-warped geometry of the Segre manifold.}}
	\label{fig_accuracy}
	\vspace{-10pt}
\end{figure}

\new{The results of our experiments are summarized in \cref{fig_accuracy}.}
We conclude that the accuracy of the consensus variants is better than a single ZIPTF.
There does not appear to be a meaningful difference between C-ZIPTF and F-ZIPTF.
\new{Moreover, the experiments suggest that F-ZIPTF is quite robust to the choice of warping factor $\frac{1}{\sqrt{5}} \le \alpha \le \sqrt{5}$.}

\new{The exponential and logarithmic maps of the $\alpha$-warped Segre manifold can be computed very efficiently. In particular, to compute all Fr\'echet means of the groups of rank-$1$ tensors in all $9$ warped geometries required $21.7$ seconds and involved $14,400$ evaluations of both the exponential and logarithmic maps. On average, this is approximately $0.75$ milliseconds for each evaluation of either exponential or logarithmic maps. This experiment was performed on a computer system comprising an AMD Ryzen 7 5800X3D CPU (8 cores, up to 3.4 GHz clock speed, 96 MB L3 cache memory) and $4\times32$ GB DDR4-3600 main memory, running Julia 1.12.4 on Ubuntu 22.04.3 LTS.}

\section*{Acknowledgments}
\new{We thank Andr\'e Uschmajew for editing our article and the two reviewers for their detailed feedback that improved this article, leading to further enhancements to the presentation of our results. We are especially grateful for the suggestions leading to clarifications in \cref{sec_applications} and the inclusion of the left panel in \cref{fig_alphawarping}, \cref{lem_basic_geodesic_form}, a shorter proof of \cref{prop_distance}, \cref{prop_geo_connected2}, and an expanded version of \cref{fig_accuracy}. The suggestions of another anonymous reviewer, which led to clarifications about Riemannian geometry in \cref{sub:warped_geometries,sec_sub_SV}, a reformulation of our results in terms of exponential and logarithmic maps rather than geodesics in \cref{prop_geodesics,thm_logarithm_preSV}, and a shorter non-constructive proof of \cref{thm_logarithm_preSV}, are also acknowledged.}

\appendix
\section{Proof of the technical results}\label{sec:proof_of_lemmas}

\subsection{Proof of \cref{lem_riemannian_covering}}\label{sub:proof_of_riemannian_covering}
Each $\imath_\sigma$ is an isometry because $\Var{P}_\alpha^\vect{k}$ is a product manifold where $\imath_\sigma$ acts only on the spheres and the maps $\mathrm{Id}$ and $-\mathrm{Id}$ are basic isometries of the sphere with the standard Euclidean inner product, and hence also with any scaled inner product.

Normality, i.e., $\otimes(p)=\otimes(q)$ implies there exists a $\sigma$ such that $q=\imath_\sigma(p)$ \cite[Appendix A]{ONeill1983}, follows from the multilinearity of the tensor product $\otimes$ \cite[Chapter 1]{Greub1978}.

By the same arguments as \cite[Section 4.1]{BBV2019}, $\otimes$ is a smooth covering map. Hence, $\dim\Var{P}_\alpha^\vect{k}=\dim\Var{S}_\alpha^\vect{k}$ and the differential $\deriv \otimes : \Tang \Var{P}_\alpha^\vect{k} \to \Tang \Var{S}_\alpha^\vect{k}$ is everywhere left-invertible.

The key part is showing that $\otimes:\Var{P}_\alpha^\vect{k}\to\Var{S}_\alpha^\vect{k}$ is Riemannian, i.e., the metric of $\Var{S}_\alpha^\vect{k}$ is the pushforward of the one of $\Var{P}_\alpha^\vect{k}$, which is a straightforward computation.
Let $\dot{\vect{u}}_i \in \Tang_{\vect{u}_i} \Sp^{n_1-1} = \vect{u}_i^\perp$ and $\dot{\lambda}\in\RR$ be arbitrary.
By differentiating $\otimes:\Var{P}_\alpha^\vect{k}\to\Var{S}_\alpha^\vect{k}$, we get
\begin{multline*}
(\deriv_{(\lambda, \vect{u}_1,\ldots,\vect{u}_d)} \otimes)(\dot{\lambda}, \dot{\vect{u}}_1,\ldots,\dot{\vect{u}}_d)\\
=
\dot{\lambda} \tensor{U} +
\lambda ( \dot{\nu}_{k_1}(\dot{\vect{u}}_1) \otimes\vect{u}_2^{\otimes k_2}\otimes\cdots\otimes \vect{u}_d^{\otimes k_d}) + \cdots + \lambda (\vect{u}_1^{\otimes k_1}\otimes\cdots\otimes\vect{u}_{d-1}^{\otimes k_{d-1}}\otimes \dot{\nu}_{k_d}(\dot{\vect{u}}_d) ),
\end{multline*}
where $\tensor{U}$ and $\nu_k(\vect{u})$ are as in \cref{eqn_U,eqn_Nu}.
The $\alpha$-warped metric on $\Var{S}_\alpha^\vect{k}$ satisfies
\begin{equation*}
 \varsigma^{\alpha}_{\lambda \tensor{U}}\left( (\deriv \otimes)(\dot{x}, \dot{\vect{u}}_1,\ldots,\dot{\vect{u}}_d), (\deriv \otimes)(\dot{y}, \dot{\vect{v}}_1,\ldots,\dot{\vect{v}}_d) \right) =
 \dot{x} \dot{y} + (\alpha \lambda)^2 \langle \dot{\tensor{U}}, \dot{\tensor{V}} \rangle,
\end{equation*}
having dropped the subscript $(\lambda, \vect{u}_1,\ldots,\vect{u}_d)$ of the differential and
where
\begin{align*}
\dot{\tensor{U}} &= (\dot{\nu}_{k_1}(\dot{\vect{u}}_1)\otimes\vect{u}_2^{\otimes k_2}\otimes\cdots\otimes\vect{u}_d^{\otimes k_d}) +\cdots+ (\vect{u}_1^{\otimes k_1}\otimes\cdots\otimes\vect{u}_{d-1}^{\otimes k_{d-1}}\otimes\dot{\nu}_{k_d}(\dot{\vect{u}}_d)),\\
\dot{\tensor{V}} &= (\dot{\nu}_{k_1}(\dot{\vect{v}}_1)\otimes\vect{u}_2^{\otimes k_2}\otimes\cdots\otimes\vect{u}_d^{\otimes k_d}) +\cdots+ (\vect{u}_1^{\otimes k_1}\otimes\cdots\otimes\vect{u}_{d-1}^{\otimes k_{d-1}}\otimes\dot{\nu}_{k_d}(\dot{\vect{v}}_d)).
\end{align*}
Recall from~\cite{Hackbusch2019} that the Euclidean inner product between rank-$1$ tensors satisfies
\(
 \langle \vect{u}_1\otimes\cdots\otimes\vect{u}_d, \vect{v}_1\otimes\cdots\otimes\vect{v}_d \rangle = \prod_{i=1}^d \langle \vect{u}_i, \vect{v}_i \rangle.
\)
Hence, rank-$1$ tensors are orthogonal to one another if they are orthogonal in at least one factor.
We compute that
\(
\langle \nu_{k_i}(\dot{\vect{u}}_i), \vect{u}_i^{\otimes k} \rangle = 0,
\)
while
\(
\langle \nu_{k_i}(\dot{\vect{u}}_i), \nu_{k_i}(\dot{\vect{v}}_i) \rangle = k_i \langle \dot{\vect{u}}_i, \dot{\vect{v}}_i \rangle
\)
if $\dot{\vect{u}}_i, \dot{\vect{v}}_i \in \Tang_{\vect{u}_i} \Sp^n$.
As a consequence, we can observe that the cross terms in $\langle\dot{\tensor{U}},\dot{\tensor{V}}\rangle$ vanish, so we are left with
\begin{align*}
\langle\dot{\tensor{U}},\dot{\tensor{V}}\rangle
= \sum_{i=1}^d \langle \dot{\nu}_{k_i}(\dot{\vect{u}}_i), \dot{\nu}_{k_i}(\dot{\vect{v}}_i) \rangle
= \sum_{i=1}^d k_i \langle \dot{\vect{u}}_i, \dot{\vect{v}}_i \rangle.
\end{align*}
This shows that $\varsigma_{\lambda\tensor{U}}$ is the pushforward of $g_{\lambda, \vect{u}_1,\ldots,\vect{u}_d}$ under $\otimes$, concluding the proof.
\qed

\subsection{Proof of \cref{lem_basic_covering_property}}\label{sub:proof_of_covering_property}
 Every smooth curve in $\Var{N}$ has such a unique smooth lift through a normal covering by \cite[Lemma A.9]{ONeill1983}. Moreover, as $\phi$ is a local isometry that restricts to an isometry on evenly covered neighborhoods, the length of $\gamma$ and its unique lift $\widetilde{\gamma}$ in $\Var{M}$ coincide. Indeed, if we partition the lift $\widetilde{\gamma}$ into the curve segments $\widetilde{\gamma}_i$ such that $\widetilde{\gamma}_i \subset \Var{M}_i \subset \Var{M}$ with $\phi|_{\Var{M}_i} : \Var{M}_i \to \Var{N}_i$ an isometry (such a partition exists and is used to prove the existence of a unique lift in \cite[Lemma A.9]{ONeill1983}), then
 \[
  \ell(\widetilde{\gamma})
  = \sum_{i=1}^k \ell(\widetilde{\gamma}_i)
  = \sum_{i=1}^k \ell(\phi(\widetilde{\gamma}_i) )
  = \ell(\gamma).
 \]
 This concludes the first part of the proof.

 If there were a strictly shorter smooth curve $\widetilde{\gamma}' \subset \Var{M}$ connecting the endpoints of the lift $\widetilde{\gamma}$, then by the same partitioning argument $\phi \circ \widetilde{\gamma}' \subset \Var{N}$ would be a shorter piecewise smooth curve ($\phi$ is a local embedding) than $\gamma$, contradicting its minimality. This proves the second part.
\qed

\subsection{Proof of \cref{prop_matching_algorithm}}\label{sub:proof_of_matching_algorithm}

We want to maximize $f = \sum_{i = 1}^{d} \sigma_i \Delta_i$ over $\sigma_i \in \set{-1, 1}$ with the constraint that $\sigma_1^{k_1} \cdots \sigma_d^{k_d} = 1$.

We can assume without loss of generality that all $k_i$'s are odd. Indeed, even $k_i$ do not impact the constraint, so we can choose the corresponding sign so that $|\Delta_i| = \sigma_i \Delta_i$, which is clearly optimal, also without constraint, because of the bilinear structure of $f$. The matching algorithm chooses these signs precisely in this way.

We can further assume without loss of generality that all $\Delta_i$ are nonzero, for otherwise we could take the unconstrained optimum \cref{eq:minimizing_sigmas} and simply swap the sign of a $\sigma_j$ for which $\Delta_j=0$ to obtain a feasible solution that is also globally optimal, which is what the matching algorithm does.

Let $(\varsigma_1, \dots, \varsigma_d)$ be a constrained optimizer of $f$ that differs in the least number of indices from the unconstrained optimizer $(\sigma_1,\dots,\sigma_d)$ from \cref{eq:minimizing_sigmas}. Denote the number of different indices by $p = \sharp\{ i \in \{1,\dots,d\} \mid \sigma_i \ne \varsigma_i \}$. Then, there are three cases:
\begin{enumerate}
\item[$p=0$.] The unconstrained optimum given by \cref{eq:minimizing_sigmas} and chosen by the algorithm satisfies the constraint.
 \item[$p = 1$.] If $f_0$ denotes the value of $f$ for the unconstrained optimum, such a constrained optimum has value $f_0 - 2\abs{\Delta_i}$.
	Hence, $i$ must correspond to the minimal $\abs{\Delta_i}$, which the matching algorithm indeed selects by construction.
	\item[$p \ge 2$.] Let $i \ne j$ be any two indices for which $\varsigma_i \ne \sigma_i$ and $\varsigma_j \ne \sigma_j$.
	By changing the signs of $\varsigma_i$ and $\varsigma_j$ to match those of $\sigma_i$ and $\sigma_j$, respectively, this increases the objective by $2|\Delta_i| + 2|\Delta_j|$. Since $|\Delta_i|, |\Delta_j| > 0$, this modified, feasible constrained point with $p-2$ different indices has a strictly higher objective value than $(\varsigma_1,\dots,\varsigma_d)$, which was supposed to be the constrained \new{optimizer} with minimal $p$. This is a contradiction, so this case cannot occur.
\end{enumerate}
This concludes the proof.
\qed

\subsection{Proof of \cref{eqn_reparameterization}}%
\label{sub:Proof of eqn_reparameterization}

Let \(C = \frac{\dot{\lambda}}{\alpha \lambda N}\), and recall that \(B = \mathrm{arctan}(C)\) and \(A = \lambda \cos(B) = \frac{\lambda}{\sqrt{1 + C^2}}.
\)
Define
\begin{align}\label{eqn_def_z}
 z(t) := \frac{\lambda + \dot{\lambda}t}{\alpha \lambda N t} = \frac{1}{\alpha N t} + C, \quad\text{so that}\quad z'(t) = -\frac{1}{\alpha N t^2}.
\end{align}
In this notation, \cref{eqn_reparameterization} becomes
\(
 \alpha s(t) = \frac{\pi}{2} - \arctan(z(t)),
\)
so that
\[
 \alpha s'(t) = -\frac{z'(t)}{1 + z^2(t)}.
\]
Since $\sqrt{ \dot{\lambda}^2 + (\alpha \lambda N)^2} = \alpha \lambda N \sqrt{1 + C^2}$, \cref{eq:ivp} simplifies as follows:
\begin{align*}
 \alpha s'(t)
 &= \cos^2(\alpha s(t) + B) A^{-1} \sqrt{ \dot{\lambda}^2 + (\alpha \lambda N)^2}\\
 &= \cos^2(\alpha s(t) + B) \alpha N (1 + C^2)\\
 &= \alpha N (1 + C^2) \cos^2 \Bigl( \frac{\pi}{2} - (\arctan(z(t)) - \arctan(C)) \Bigr)\\
 &= \alpha N (1 + C^2) \sin^2( \arctan(z(t)) - \arctan(C) ).
\end{align*}
Using basic trigonometric identities, it is easy to obtain the next identity:
\begin{align}\label{eqn_sin_appendix}
\sin(\arctan(x) - \arctan(y))=\frac{x-y}{\sqrt{1+x^2} \sqrt{1+y^2}}.
\end{align}
This allows us to conclude the proof because
\[
 \alpha s'(t)
 = \alpha N (1 + C^2) \frac{(z(t) - C)^2}{(1 + z^2(t))(1 + C^2)}
 = \alpha N \frac{(\alpha N t)^{-2}}{1 + z^2(t)}
 = -\frac{z'(t)}{1 + z^2(t)}.
\]

\subsection{Proof of \cref{eq:gamma_of_1}}%
\label{sub:Proof of eq:gamma_of_1}

Using the same notation as in \cref{sub:Proof of eqn_reparameterization}, we get that
\(
 \alpha s(1) = \frac{\pi}{2} - \arctan( z(1) ),
\)
so that
\begin{align*}
\gamma(1)
= \mu(s(1))
= \frac{\lambda \cos( B )}{\cos(\alpha s(1) + B)}
&= \frac{ \lambda }{\sqrt{1 + C^2}} \cdot \frac{1}{\sin(\arctan(z(1)) - \arctan(C))}\\
&= \frac{ \lambda }{\sqrt{1 + C^2}} \frac{ \sqrt{1+C^2} \sqrt{1 + z^2(1)} }{ z(1) - C }\\
&= \alpha \lambda N \sqrt{1 + z^2(1)},
\end{align*}
where we used \cref{eqn_sin_appendix} in the penultimate and \cref{eqn_def_z} in the final equality.
Since $z^2(1) = (\lambda + \dot{\lambda})^2 (\alpha \lambda N)^{-2}$, the proof is concluded.

\bibliographystyle{siam}
\bibliography{SJVV}

\end{document}